\documentclass[a4paper]{amsart}

\usepackage{amssymb}

\usepackage{mathdots}
\newtheorem{thm}{Theorem}[section]
\newtheorem{cor}[thm]{Corollary}
\newtheorem{lem}[thm]{Lemma}
\newtheorem{prop}[thm]{Proposition}

\theoremstyle{remark}

\theoremstyle{definition}
\newtheorem{defi}[thm]{Definition}
\newtheorem{examp}[thm]{Example}

\begin{document}
\pagestyle{myheadings}
\address{Department of Mathematics, Technion - Israel Institute of Technology, 32000, Haifa, Israel.}
\email{max@tx.technion.ac.il}
\title{Subproduct systems over $\mathbb{N}\times\mathbb{N}$}
\author{Maxim Gurevich}
\date{\today}

\begin{abstract}
We develop the theory of subproduct systems over the monoid $\mathbb{N}\times \mathbb{N}$, and the non-self-adjoint operator algebras associated with them. These are double sequences of Hilbert spaces $\{X(m,n)\}_{m,n=0}^\infty$ equipped with a multiplication given by coisometries from $X(i,j)\otimes X(k,l)$ to $X(i+k, j+l)$. We find that the character space of the norm-closed algebra generated by left multiplication operators (the tensor algebra) is homeomorphic to a complex homogeneous affine algebraic variety intersected with a unit ball. Certain conditions are isolated under which subproduct systems whose tensor algebras are isomorphic must be isomorphic themselves. In the absence of these conditions, we show that two numerical invariants must agree on such subproduct systems.\\
Additionally, we classify the subproduct systems over $\mathbb{N}\times \mathbb{N}$ by means of ideals in algebras of non-commutative polynomials.
\end{abstract}

\maketitle

\section{Introduction}
A subproduct system over a cancellative abelian monoid $\mathcal{S}$ is roughly a collection of finite-dimensional Hilbert spaces $\{X(s)\}_{s\in\mathcal{S}}$ equipped with coisometries $\{U_{s,t}:X(s)\otimes X(t)\to X(s+t)\}_{s,t\in\mathcal{S}}$ that behave in an associative fashion. This is a formal way of seeing $X(s+t)$ as a subspace of $X(s)\otimes X(t)$, for all $s,t\in \mathcal{S}$. When all $\{U_{s,t}\}$ are injective, this structure is called a product system.

Each vector $\eta\in X(t)$ defines, for each $s\in\mathcal{S}$, a shift from $X(s)$ to $X(s+t)$ by left tensoring with $\eta$. These shifts may all be represented on a suitable Hilbert space, and thus give rise to (several) operator algebras. Much of the interest in subproduct systems lies in the study of these operator algebras. 

The concept of a subproduct system was explicitly introduced by Shalit and Solel in \cite{solel-orr} as a tool for the study of cp-semigroups. Their emphasis was on the case of $\mathcal{S}=\mathbb{N}=\{0,1,2,\ldots\}$. It was seen that in this case the norm-closed algebra generated by the shifts (tensor algebra) serves as a universal object with regard to tuples of operators subject to given non-commutative relations. Another similar universality property was shown for the weakly closed algebra in \cite{gur-meyer}.

Independently, subproduct systems have appeared under the name `inclusion systems' in \cite{inclusion}. Also, we need to mention that subproduct systems and product systems are very often studied in a more general context in which the finite dimensional Hilbert spaces are replaced with $C^\ast$-correspondences. In this paper we restrict ourselves to the Hilbert space setting.

It was established that already in the case of the elementary monoid $\mathbb{N}$, subproduct systems and their associated operator algebras form a rich class of objects, which pose many research questions. Hence, it should come as no surprise, that subproduct systems over monoids with more complicated structure, e.g. $\mathbb{N}^k$, are quite an obscure domain. The aim of this paper is to develop a theory of subproduct systems for the case of $\mathcal{S}=\mathbb{N}\times \mathbb{N}$.

Our motivation and methods come from two closely related lines of study which have been developing in recent years. One is the aforementioned study of the non-self-adjoint algebras arising from subproduct systems over $\mathbb{N}$. From this direction, we draw the perspective from the initial work in \cite{solel-orr} and the consequent work of Davidson, Ramsey and Shalit \cite{orr-subprod}. The latter paper examined to what extent the tensor algebra of a subproduct system is a complete invariant of the subproduct system. We will attack the same question in our setting of $\mathbb{N}^2$.

The second direction is the study of product systems over $\mathbb{N}^2$. In \cite{solel-unitary}, the structure of the non-self-adjoint algebras associated with such a product system was analyzed, while \cite{dav-rep} concentrated on the representation theory of these algebras. This theme has a relation to higher-rank graph algebras in the sense of \cite{highrank}, since a rank $2$ graph on one vertex gives a special kind of a product system over $\mathbb{N}^2$ (see, for example \cite{highrank-one}).

We largely follow the paradigm that an operator algebra generated by the shifts of a subproduct system should behave as \textit{a quotient} of an algebra generated by the shifts of a product system. Indeed, it was implicit in the results of \cite{solel-orr} that tensor algebras of subproduct systems over $\mathbb{N}$ are isometrically isomorphic to quotients of corresponding product system tensor algebras. Later, this result was explicitly expanded to more general frameworks such as subproduct systems over $C^\ast$-correspondences \cite{ami}, and weakly closed algebras \cite{gur-meyer} (following the more general description of quotients of weakly closed algebras in \cite{dav-distance} and \cite{popescu-distance}).

Our main objects of interest will be the tensor algebra $\mathcal{A}_X$ associated with a subproduct system $X$ over $\mathbb{N}^2$, that is, the norm closed algebra generated by the shifts of $X$.

In Section 3, we show that in similarity with the situation over $\mathbb{N}$, each subproduct system can be embedded into a matching product system over $\mathbb{N}^2$. Next, we show that on the level of one-dimensional representations, $\mathcal{A}_X$ indeed behaves as a quotient of a matching product system. In other words, the character space of $\mathcal{A}_X$ is composed of characters on the tensor algebra of the product system that vanish on a certain ideal. That result, combined with the results of Power and Solel in \cite{solel-unitary} on product systems, gives an interesting description of the character space of $\mathcal{A}_X$: As a topological space, it is homeomorphic to a homogeneous affine algebraic variety in $\mathbb{C}^n$ (with $n= \dim X(1,0)+\dim X(0,1)$), intersected with a set that may be viewed as unit ball of a certain fixed norm on $\mathbb{C}^n$.

Section 4 slightly deviates from the main line of study to present an algebraic point of view on subproduct systems over $\mathbb{N}^2$. We are able to parametrize the collection of subproduct systems that can be embedded in a given product system, by a certain collection of non-commutative polynomials. In this manner we produce the existence of an abundance of subproduct systems.

In Section 5, we strive to apply the newly acquired knowledge about the character space, to the isomorphism problem for tensor algebras. Given two such algebras $\mathcal{A}_X$, $\mathcal{A}_Y$ that are isometrically isomorphic, does it necessarily mean that the underlying subproduct systems $X$ and $Y$ are isomorphic? Recalling that an isomorphism of the algebras induces a homeomorphism of their character spaces, one might be motivated to search for a natural condition, which could be formulated in terms of the action of this homeomorphism. Indeed, we will isolate a large class of subproduct systems for which the existence of a certain fixed point of the homeomorphism implies a positive answer to the above question.

In the last section, we look into subproduct systems whose tensor algebras are isomorphic, without further assumptions. In general, subproduct systems as such need not be isomorphic themselves. Yet, by looking at the differential structure of the tensor algebra's character space, we manage to arrive at two numerical invariants that must coincide for these subproduct systems.

This work is based on part of the author's M.~Sc.~thesis.



\section{Definitions}
\begin{defi}
Let $\mathcal{S}$ be an additive cancellative abelian monoid. A \textit{subproduct system (of finite-dimensional Hilbert spaces) over $\mathcal{S}$} is a family $X = \{X(s)\}_{s\in \mathcal{S}}$ of finite-dimensional Hilbert spaces such that \\
(1) $\dim X(0)= 1$ \\
(2) For every $s,t\in \mathcal{S}$ there is a coisometry 
$$U_{s,t}: X(s)\otimes X(t) \to X(s+t)$$
(3) The maps $\{U_{s,0}\}_{s\in \mathcal{S}}$ and $\{U_{0,s}\}_{s\in \mathcal{S}}$ are given by the natural isomorphisms $\mathbb{C}\otimes X(s)\cong X(s)\cong X(s)\otimes \mathbb{C}$. \\
(4) The maps $\{U_{s,t}\}_{s,t\in \mathcal{S}}$ satisfy
$$ U_{s+t,r}(U_{s,t}\otimes I_{X(r)}) = U_{s,t+r} (I_{X(s)}\otimes U_{t,r}) \quad \forall s,t,r\in \mathcal{S}$$
\end{defi}

Given a subproduct system $X$ over $\mathcal{S}$, each $s,t\in \mathcal{S}$ and $\eta\in X(s)$ define a \textit{creation operator} 
$$L_{\eta}^{s,t}:X(t)\to X(s+t)\quad \xi\:\mapsto U_{s,t}(\eta\otimes \xi)$$
Next, we define the Hilbert space $\mathcal{F}_X:= \bigoplus_{s\in\mathcal{S}} X(s)$, and call it the \textit{$X$-Fock space}.

For all $s\in \mathcal{S}$ and $\eta\in X(s)$, let $L^s_\eta$ be the linear operator on $\mathcal{F}_X$ defined on the summands by
$$L^{(s)}_\eta|_{X(t)}:= L^{s,t}_\eta\quad \forall t\in \mathcal{S}$$
\begin{lem}\label{norm-eq}
If $X$ is a subproduct system, then for all $\eta\in X(s)$, the operator $L^{(s)}_\eta$ is bounded, and its norm equals $\|\eta\|$.
\end{lem}
\begin{proof}
Denote $\Delta =1\in X(0)$. First, we have $\|\eta \|= \|L^{(s)}_\eta\Delta\|\leq \|L^{(s)}_\eta\|$. But, also, for all $\xi\in X(t)$,
$$\|L^{s,t}_\eta(\xi)\| = \|U^X_{s,t}(\eta\otimes \xi)\|\leq \|\eta\|\|\xi\|$$
Considering the fact that the ranges of $\{L^{s,t}_\eta\}_{t\in \mathcal{S}}$ are the orthogonal subspaces $\{X(s+t)\}_{t\in\mathcal{S}}$ ($\mathcal{S}$ is cancellative), we can conclude that $\|L^{(s)}_\eta\|=\sup_{t\in\mathcal{S}} \|L^{s,t}_\eta\|\leq \|\eta\|$. 
\end{proof}

\begin{defi}
We call the norm closed algebra $\mathcal{A}_X\subset B(\mathcal{F}_X)$, which is generated by $\{L^{(s)}_\eta\}_{\eta\in X(s), s\in\mathcal{S}}$, the \textit{tensor algebra of $X$}.
\end{defi}

\begin{lem}\label{finite-gen}
Suppose $\mathcal{S}$ is generated as a monoid by a set $B\subset \mathcal{S}$. Then, for a subproduct system $X$ over $\mathcal{S}$, we have
$$\mathcal{A}_X = \overline{Alg}\left\{ L^{(s)}_\eta\right\}_{\eta\in X(s), s\in B}$$
\end{lem}
\begin{proof}
Suppose $\eta_0\in X(s_0)$ for some $s_0\in \mathcal{S}$. It suffices to prove that $L^{(s_0)}_{\eta_0}\in Alg\left\{ L^{(s)}_\eta\right\}_{\eta\in X(s), s\in B}$. Since $B$ generates $\mathcal{S}$, $s_0= \sum_{i=1}^n s_i$ for some $s_1,\ldots, s_n\in B$. We will prove it by induction on $n$. For $n=1$, $s_0=s_1\in B$, and there is nothing to prove. Denote $u= \sum_{i=1}^{n-1} s_i$. So, $\eta_0 = U_{s_n,u}(\sum_{j=1}^k \eta_j\otimes \zeta_j)$ for some $\{\eta_j\}\subset X(s_n)$ and $\{\zeta_j\}\subset X(u)$. Thus, for all $t\in \mathcal{S}$ and $\xi\in X(t)\subset \mathcal{F}_X$, we have,
$$L^{(s_0)}_{\eta_0}\xi = U_{s_0,t}(\eta_0\otimes \xi) = U_{s_0,t}(U_{s_n,u}\otimes I_{X(t)})\left(\sum_{j=1}^k \eta_j\otimes \zeta_j\otimes \xi\right)=$$
$$ =U_{s_n,u+t}(I_{X(s_n)}\otimes U_{u,t})\left(\sum_{j=1}^k \eta_j\otimes \zeta_j\otimes \xi\right) = U_{s_n,u+t} \left(\sum_{j=1}^k \eta_j\otimes L^{(u)}_{\zeta_j}\xi\right) = $$ $$=\sum_{j=1}^k L^{(s_n)}_{\eta_j}L^{(u)}_{\zeta_j}\xi \quad \Rightarrow \quad L^{(s_0)}_{\eta_0} = \sum_{j=1}^k L^{(s_n)}_{\eta_j}L^{(u)}_{\zeta_j}$$
But, by the induction hypothesis $\{L^{(u)}_{\zeta_j}\}$ are in the algebra generated by $\left\{ L^{(s)}_\eta\right\}_{\eta\in X(s), s\in B}$, hence, so is $L^{(s_0)}_{\eta_0}$.
\end{proof}
\begin{cor}
For a subproduct system $X$ over $\mathbb{N}^2$, the tensor algebra $\mathcal{A}_X$ is the unital norm-closed algebra generated by the operators $\{L^{(1,0)}_\eta\}_{\eta\in X(1,0)}$ and $\{L^{(0,1)}_\eta\}_{\eta\in X(0,1)}$.
\end{cor}
From now on, unless stated otherwise, we will always assume $\mathcal{S} = \mathbb{N}\times \mathbb{N}$, and when referring to a subproduct system the assumption will be it is over $\mathbb{N}\times\mathbb{N}$.

We will also wish to have a notion of an isomorphism between subproduct systems. The following definition is new because it may incorporate an automorphism of the underlying monoid, that is, a coordinate switch.
\begin{defi}
Suppose $X,Y$ are two subproduct systems over $\mathbb{N}^2$, with families of coisometries $\{U^X_{s,t}\}$ and $\{U^Y_{s,t}\}$. We say that $X$ and $Y$ are \textit{isomorphic}, if there exists a collection of unitary maps $\{V_s:X(s)\to Y(\phi(s))\}_{0\not=s\in \mathbb{N}^2}$ such that 
$$V_{s+t}\circ U^X_{s,t} = U^Y_{\phi(s),\,\phi(t)}\circ (V_s\otimes V_t)\quad \forall s,t\in\mathcal{S}$$
where $\phi:\mathbb{N}^2\to\mathbb{N}^2$ is either the identity or the coordinate switch $\phi(i,j)=(j,i)$. 
\end{defi}

\section{From product systems to subproduct systems}
Given a subproduct system $X$, we would like to study the properties of the tensor algebra $\mathcal{A}_X$. One such property is its character space $\mathcal{M}(\mathcal{A}_X)$, that is, the set of all multiplicative functions $\alpha: \mathcal{A}_X\to \mathbb{C}$. One may view it as the study of one-dimensional representations of $\mathcal{A}_X$. 

In order to improve the understanding of $\mathcal{M}(\mathcal{A}_X)$, we need to translate our setting to the setting discussed in \cite{solel-unitary}. We will briefly review it now, and delve into greater detail a while later. Given two finite-dimensional Hilbert spaces $E,F$, we construct the full Fock space $\mathcal{F}(E,F):= \bigoplus_{m,n=0}^\infty E^{\otimes m}\otimes F^{\otimes n}$ (with $E^{\otimes 0}\otimes F^{\otimes 0} = \mathbb{C}$). The construction in \cite{solel-unitary}, assigns to each unitary operator $u: F\otimes E\to E\otimes F$ a norm-closed algebra $\mathcal{A}_u\subset B(\mathcal{F}(E,F))$. This algebra is generated by two tuples of operators, $L_{e_1},\ldots, L_{e_m}$ and $L_{f_1},\ldots, L_{f_n}$, each of which is freely noncommuting, and when combined they are subject to the relations
$$L_{f_j}L_{e_i} = \sum_{k,l} u_{(k,l),(i,j)}L_{e_k}L_{f_l}$$
where $(u_{(k,l),(i,j)})\in M_{mn}(\mathbb{C})$ is a unitary matrix, which represents $u$ relative to a suitable basis.

We will show that for every subproduct system $X$, the algebra $\mathcal{A}_X$ is continuously isomorphic to the \textit{norm-closure} of the compression of $\mathcal{A}_u\subset B(\mathcal{F}(E,F))$ onto a co-invariant subspace, for $E=X(1,0)$, $F=X(0,1)$ and $u$ a commutation relation for $X$. 

\subsection{Standard subproduct systems}
Recall, that \cite{solel-orr} had a notion of a standard subproduct system over $\mathbb{N}$. That was a subproduct system over $\mathbb{N}$ which was given by a sequence of subspaces $\{X(n)\subset X(1)^{\otimes n}\}_{n=1}^\infty$ and the coisometries were naturally given by the projections onto those subspaces. It was shown there that, up to isomorphism, all subproduct systems over $\mathbb{N}$ are standard. Proposition \ref{standard-2} will show that a similar result is valid in our setting. That is, the elements of every subproduct system (in our current context) $X$ can be seen as subspaces of tensor products of $X(1,0)$ and $X(0,1)$. More precisely, if we denote $E=X(1,0)$ and $F=X(0,1)$, we will show that, up to isomorphism, for all $(m,n)$, $X(m,n)\subset E^{\otimes m}\otimes F^{\otimes n}$ and thus, $\mathcal{F}_X\subset \mathcal{F}(E,F)$, and the coisometries are given by some agreed multiplication in $\mathcal{F}(E,F)$ and projection onto $\mathcal{F}_X$. We will call such product system \textit{standard}.

It is worth remarking, that in \cite{solel-orr} it was displayed that already over the monoid $\mathbb{N}^3$ such property \textit{does not} hold. By that we mean, that, vaguely speaking, not every subproduct system $X$ over $\mathbb{N}^3$ can be seen as a sub-structure of the tensor products of $X(1,0,0),X(0,1,0),X(0,0,1)$. This fact can be taken as the initial motivation for specializing the study to subproduct system over $\mathbb{N}\times \mathbb{N}$. 

Before stating precisely the aforementioned proposition, we need to describe a simple construction. Suppose $E$ and $F$ are Hilbert spaces, and $u:F\otimes E\to E\otimes F$ an operator. We would like to extend this `commutation' operator in a natural manner to an operator from $ E^{\otimes i}\otimes F^{\otimes j}\otimes E^{\otimes k}\otimes F^{\otimes l}$ to $E^{\otimes i+k}\otimes F^{\otimes j+l}$. This is done as follows. For all $n\geq1$, denote $u^{(1,n)}: F\otimes E^{\otimes n}\to E^{\otimes n}\otimes F$ to be the operator given by
$$u^{(1,n)} =(I_{E^{\otimes (n-1)}}\otimes u)(I_{E^{\otimes (n-2)}}\otimes u\otimes I_{E})\cdots(u\otimes I_{E^{\otimes (n-1)}}) $$
Next, denote for all $m,n\geq 1$, $u^{(m,n)}:F^{\otimes m}\otimes E^{\otimes n} \to E^{\otimes n}\otimes F^{\otimes m}$ to be the operator given by
$$u^{(m,n)} = (u^{(1,n)}\otimes I_{F^{\otimes m-1}})\cdots (I_{F^{\otimes (m-2)}}\otimes u^{(1,n)}\otimes I_F)(I_{F^{\otimes (m-1)}}\otimes u^{(1,n)})$$
Finally, denote,
$$W_{(i,j),(k,l)}^u := I_{E^{\otimes i}}\otimes u^{(j,k)}\otimes I_{F^{\otimes l}}: E^{\otimes i}\otimes F^{\otimes j}\otimes E^{\otimes k}\otimes F^{\otimes l}\to  E^{\otimes i+k}\otimes F^{\otimes j+l} $$
Note, that when $u$ is taken to be unitary, $W^u_{(i,j),(k,l)}$ is unitary.

\begin{prop}\label{standard-2}
Let $X = \{X(m,n)\}_{(m,n)\in \mathbb{N}^2}$ be a subproduct system with coisometries $\{U^X_{(i,j),(k,l)}\}$. Then, $X$ is isomorphic to a subproduct system $Y = \{Y(m,n)\}_{(m,n)\in \mathbb{N}^2}$ with coisometries $\{U^Y_{(i,j),(k,l)}\}$ such that for all $(m,n)\not= (0,0)$,
$$Y(m,n) \subset  E^{\otimes m}\otimes F^{\otimes n}$$
where $E=X(1,0)$ and $F=X(0,1)$, and 
$$U^Y_{(i,j),(k,l)} = p_{(i+k, j+l)}W_{(i,j),(k,l)}^u \;\arrowvert_{Y(i,j)\otimes Y(k,l)}$$
where $p_{(m,n)}\in B(E^{\otimes m}\otimes F^{\otimes n})$ are the orthogonal projection on $Y(m,n)$, and $u$ can be chosen as any unitary operator from $F\otimes E$ to $ E\otimes F$ that satisfies $p_{(1,1)}u = (U^{X}_{(1,0),(0,1)})^\ast U^X_{(0,1),(1,0)}$.
\\
Moreover, we have,
$$p_{(i+k,j+l)}\leq W_{(i,j),(k,l)}^u(p_{(i,j)}\otimes I_{E^{\otimes k}\otimes F^{\otimes l}})(W_{(i,j),(k,l)}^u)^\ast$$
$$p_{(i+k,j+l)}\leq W_{(i,j),(k,l)}^u(I_{E^{\otimes i}\otimes F^{\otimes j}}\otimes p_{(k,l)})(W_{(i,j),(k,l)}^u)^\ast$$

\end{prop}

\begin{proof}
For every tuple of non-negative integers $\overline{v} = (m_1,n_1,\ldots,m_r,n_r)$, let 
$$V_{\overline{v}}:E^{\otimes m_1}\otimes F^{\otimes n_1}\otimes \cdots \otimes E^{\otimes m_r}\otimes F^{\otimes n_r}  \to X\left(\sum_{i=1}^r m_i,\sum_{i=1}^r n_i\right)$$
be the coisometry which is a composition of maps from the collection $\{U^X_{(i,j),(k,l)}\}$. It is a consequence of the associativity properties of $X$ that an operator as such is well-defined. Define $Y(m,n)$ to be the initial space of $V_{(m,n)}$. Suppose $u$ is any unitary operator that satisfies the requirement in the statement. Notice, that
$$V_{(1,1)}u = V_{(1,1)}p_{(1,1)}u =  U^{X}_{(1,0),(0,1)}(U^{X}_{(1,0),(0,1)})^\ast U^X_{(0,1),(1,0)} = U^X_{(0,1),(1,0)} = V_{(0,1,1)}$$
Now, by induction, one can show that $V_{(m,n)}u^{(m,n)} = V_{(0,n,m)}$. That, in turn, means that
$$V_{(i,j,k,l)} = U^X_{(i,0),(k,j+l)}(I_{X(i,0)}\otimes U^X_{(k,j),(0,l)})(V_{(i)}\otimes V_{(0,j,k)}\otimes V_{(0,l)}) =$$
$$=  U^X_{(i,0),(k,j+l)}(I_{E^{\otimes i}}\otimes U^X_{(k,j),(0,l)})(V_{(i)}\otimes V_{(k,j)}\otimes V_{(0,l)})(I_{E^{\otimes i}}\otimes u^{(j,k)}\otimes I_{F^{\otimes l}})=$$
$$= V_{(i+k,j+l)}W^u_{(i,j),(k,l)} $$
Therefore, $p_{(i+k,j+l)}W_{(i,j),(k,l)}^u = (V_{(i+k,j+l)})^\ast V_{(i,j,k,l)}$. Thus, $\{U^Y_{(i,j),(k,l)}\}$ are indeed coisometries. Also, since $V_{(i,j,k,l)}= U^X_{(i,j),(k,l)}(V_{(i,j)}\otimes V_{(k,l)})$, we have 
$$p_{(i+k,j+l)}W_{(i,j),(k,l)}^u(I\otimes p_{(k,l)}) = p_{(i+k,j+l)}W_{(i,j),(k,l)}^u(p_{(i,j)}\otimes I)=p_{(i+k,j+l)}W_{(i,j),(k,l)}^u $$
This easily gives the inequalities written in the statement. 

To show the associativity of those maps, we will compute $V_{(i,j,k,l,s,t)}$ in two ways.
$$V_{(i,j,k,l,s,t)} = U^X_{(i+k,j+l),(s,t)}(V_{(i,j,k,l)}\otimes V_{(s,t)}) =$$ $$ =U^X_{(i+k,j+l),(s,t)}(V_{(i+k,j+l)}\otimes V_{(s,t)}) (W^u_{(i,j),(k,l)}\otimes I_{E^{\otimes s}\otimes F^{\otimes t}})= $$ $$= V_{(i+k,j+l,s,t)}(W^u_{(i,j),(k,l)}\otimes I_{E^{\otimes s}\otimes F^{\otimes t}})=$$ $$= V_{(i+k+s,j+l+t)}U^Y_{(i+k,j+l),(s,t)}(W^u_{(i,j),(k,l)}\otimes I_{E^{\otimes s}\otimes F^{\otimes t}})$$
Similarly,
$$V_{(i,j,k,l,s,t)} = U^X_{(i,j),(k+s,l+t)}(V_{(i,j)}\otimes V_{(k,l,s,t)}) =$$ $$ =U^X_{(i,j),(k+s,l+t)}(V_{(i,j)}\otimes V_{(k+s,l+t)}) (I_{E^{\otimes i}\otimes F^{\otimes j}}\otimes W^u_{(k,l),(s,t)})= $$ $$= V_{(i,j, k+s,l+t)}(I_{E^{\otimes i}\otimes F^{\otimes j}}\otimes W^u_{(k,l),(s,t)})=$$ $$= V_{(i+k+s,j+l+t)}W^u_{(i,j),(k+s,l+t)}(I_{E^{\otimes i}\otimes F^{\otimes j}}\otimes W^u_{(k,l),(s,t)})$$
Hence, we can conclude that,
$$U^Y_{(i+k,j+l),(s,t)}(U^Y_{(i,j),(k,l)}\otimes I_{Y(s,t)}) = U^Y_{(i,j),(k+s,l+t)}(I_{Y(i,j)}\otimes U^Y_{(k,l),(s,t)})$$
We will finish by showing that the maps $V_{(m,n)}^\ast:X(m,n)\to Y(m,n)$ constitute an isomorphism of subproduct systems. Indeed, for all $x_1\in X(i,j)$ and $x_2\in X(k,l)$, 
$$ V_{(i+k,j+l)} U^Y_{(i,j),(k,l)}(V_{(i,j)}^\ast(x_1)\otimes V_{(k,l)}^\ast(x_2)) = V_{(i,j,k,l)}(V_{(i,j)}^\ast(x_1)\otimes V_{(k,l)}^\ast(x_2)) =$$ $$= U^X_{(i,j),(k,l)}(V_{(i,j)}V_{(i,j)}^\ast(x_1)\otimes V_{(k,l)}V_{(k,l)}^\ast(x_2)) = U^X_{(i,j),(k,l)}(x_1\otimes x_2) $$
$$\Longrightarrow\quad U^Y_{(i,j),(k,l)}(V_{(i,j)}^\ast(x_1)\otimes V_{(k,l)}^\ast(x_2)) = V_{(i+k,j+l)}^\ast U^X_{(i,j),(k,l)}(x_1\otimes x_2)$$

\end{proof}
\textit{Remark:} It is also possible to prove that \textit{any} double-sequence of subspaces $Y(m,n)\subset E^{\otimes m}\otimes F^{\otimes n}$ as in the proposition above defines a subproduct system with coisometries as defined above, as long as it satisfies the stated inequalities of projections. One should only do the simple, yet tedious, work of verifying the associativity of the maps involved.

\begin{defi}
A subproduct system over $\mathbb{N}\times \mathbb{N}$ which satisfies the conditions of $Y$ in the statement of Proposition \ref{standard-2} will be called a \textit{standard subproduct system}. A unitary $u$ which satisfies the conditions in the statement of Proposition \ref{standard-2} will be called a \textit{commutation relation} for $Y$.
\end{defi}
\begin{cor}\label{standard}
Let $X$ be a subproduct system. Then, the tensor algebra $\mathcal{A}_X$ isometrically isomorphic to a tensor algebra $\mathcal{A}_Y$ of a standard subproduct system.
\end{cor}
\begin{proof}
Proposition \ref{standard-2} supplies a standard subproduct system $Y$, and a subproduct system isomorphism $\{V_{m,n}:X(m,n)\to Y(m,n)\}_{(m,n)\in\mathbb{N}^2}$. Notice, that $V:=\bigoplus_{(m,n)\in\mathbb{N}^2} V_{m,n}: \mathcal{F}_X\to \mathcal{F}_Y$ is unitary. A quick check can show, also, that for all $\eta\in X(m,n)$,
$$L^{(m,n)}_\eta = V^\ast L^{(m,n)}_{V_{m,n}(\eta)} V$$
Thus, the result follows from the fact that $H^\infty(X)$ and $H^\infty(Y)$ are generated by unitary equivalent sets of operators.
\end{proof}

There is, in fact, a fairly simple way of constructing standard subproduct systems. 
\begin{prop}\label{construct}
Suppose $E$ and $F$ are Hilbert spaces, $u : F\otimes E\to E\otimes F$ a unitary operator, and $L\subset \mathbb{N}\times \mathbb{N}$ is a subset which satisfies
$$(i,j)\in L\quad\Rightarrow \quad (k,l)\in L\:\forall\, k\leq i,\: l\leq j$$
Suppose we are given a set of projections $\{p_{(i,j)}\in B(E^{\otimes i}\otimes F^{\otimes j})\}_{(i,j)\in L}$ which satisfy the inequalities stated in Proposition \ref{standard-2} with respect to $u$, and $p_{(0,0)} = I_{\mathbb{C}}, p_{(1,0)}=I_E, p_{(0,1)}=I_F$.

Then,\\
\textbf{(a)} There is a maximal standard subproduct system $X$ such that $X(1,0)=E$, $X(0,1)=F$, $u$ is a commutation relation for $X$, and for all $(i,j)\in L$, $p_{(i,j)}$ is the projection on $X(i,j)$. 
The maximality is in the sense that if $Y$ is a standard subproduct system with the same commutation relation, and $Y(i,j)= X(i,j)$ for all $(i,j)\in L$, then $Y(i,j)\subset X(i,j)$ for all $(i,j)\in \mathbb{N}^2$.\\
\textbf{(b)} The fibers of $X$ are given by
$$X(i,j) = \bigcap_{\substack{(k,l)+(s,t)= (i,j)\\ (k,l)\not= (i,j)\\ (s,t)\not=(i,j)}} W^u_{(k,l),(s,t)}\left(X(k,l)\otimes X(s,t)\right) \quad \forall (i,j)\not\in L$$
\end{prop}
\begin{proof}
\textbf{(a)} Consider a collection of projections $\{p_{(i,j)}^K\in B(E^{\otimes i}\otimes F^{\otimes j})\}_{(i,j)\in K}$ such that $L\subset K\subset \mathbb{N}^2$, $K$ satisfies the condition that was required from $L$, the projections satisfy the inequalities of Proposition \ref{standard-2}, and it is maximal in the sense that $q_{(i,j)}\leq p^K_{(i,j)}$ for any $q_{(i,j)}$ which is a projection associated with $Y$ as in the statement. Also, for all $(i,j)\in L$, $p_{(i,j)}^K = p_{(i,j)}$.

One such collection is obviously $\{p_{(i,j)}\}_{(i,j)\in L}$. These collections constitute a partially ordered set with respect to inclusion. To prove what is needed we need to show there is a collection $\{p^{\mathbb{N}^2}_{(i,j)}\}_{(i,j)\in \mathbb{N}^2}$ in the poset.

It is simple to check that this poset meets the requirements of Zorn's lemma, hence, has a maximal element $\{p_{(i,j)}^{K_0}\}_{(i,j)\in K_0}$. If $K_0\not= \mathbb{N}^2$, then we can denote $i_0=\min\{i:\exists j,\,(i,j)\not\in K_0\}$ and $j_0= \min\{j:(i_0,j)\not\in K_0\}$, and define the projection
$$p'_{(i_0,j_0)} := \bigwedge_{\substack{(i,j)+(k,l)= (i_0,j_0)\\ (i,j),(k,l)\not= (i_0,j_0)}} W_{(i,j),(k,l)}^u(p_{(i,j)}^{K_0}\otimes p_{(k,l)}^{K_0})(W_{(i,j),(k,l)}^u)^\ast$$
It is left to observe that the collection $\{p_{(i,j)}^{K_0}\}_{(i,j)\in K_0}\cup \{p'_{(i_0,j_0)}\}$ belongs to the poset, in contradiction to the maximality of $K_0$. Indeed, if $Y$ is a standard subproduct system as in the statement with projections $\{q_{(i,j)}\}_{(i,j)\in \mathbb{N}^2}$, then $q_{(i,j)}\leq p_{(i,j)}^{K_0}$ for all $(i,j)\in K_0$, and so we must have $q_{(i_0,j_0)}\leq p'_{(i_0,j_0)}$. 

Thus, $\{p^{\mathbb{N}^2}_{(i,j)}\}_{(i,j)\in \mathbb{N}^2}$ define a standard subproduct system $X$ as desired.

\textbf{(b)} For ease of notation we denote the right hand side of the formula $T(i,j)$. Obviously, $X(i,j)\subset T(i,j)$. To see that the reverse inclusion is correct for a given $(i,j)\not\in L$, observe the subproduct system $Y$ constructed as follows:
$$Y(k,l) = \left\{\begin{array}{ll} X(k,l) & (k,l)\in L' \\ T(i,j) & (k,l)=(i,j) \\ \{0\} & \mbox{\small{else}}\end{array}\right.$$
where $L\subset L' := \{(k,l)\;:\; k< i\}\cup\{ (k,l)\;:\;l<j\}$. Then, the maximality property of $X$ assures $T(i,j)= Y(i,j)\subset X(i,j)$.
\end{proof}

The last proposition shows the existence of a variety of subproduct systems. For a simple example, take $E$, $F$, $u$ as we like, and $L = \{(0,0), (1,0), (2,0)\}$. Then, choose any subspace of $E^{\otimes 2}$ and denote $p_{(2,0)}$ to be the projection on it. Proposition \ref{construct} assures the existence of a maximal subproduct system $X$ with $X(2,0)$ being the chosen subspace.

\begin{examp}\label{const-from-n}
Suppose $Y_1,Y_2$ are two standard subproduct systems over $\mathbb{N}$, in the sense of \cite{solel-orr}. We present a natural way of adjoining $Y_1$ and $Y_2$ into one subproduct system over $\mathbb{N}\times \mathbb{N}$ using any given commutation relation $u$.

Denote $E = Y_1(1)$, $F= Y_2(1)$. For all $i\geq1$ let $p_{(i,0)}$ be the projection from $Y_1(1)^{\otimes i}$ onto $Y_1(i)$, and for all $j\geq1$ let $p_{(0,j)}$ be  the projection from $Y_2(1)^{\otimes j}$ onto $Y_2(j)$. Then it is easy to see that for $L = \{(i,0)\}_{i\in \mathbb{N}}\cup\{(0,j)\}_{j\in \mathbb{N}}$, and any $u$, the defined set $\{p_{(i,j)}\}_{(i,j)\in L}$ satisfies the conditions of Proposition \ref{construct}. So, by that proposition we know there is a maximal subproduct system $X$ over $\mathbb{N}\times \mathbb{N}$ with commutation relation $u$, such that $X(i,0)= Y_1(i)$ and $X(0,j)= Y_2(j)$ for all $i,j\geq1$.
\end{examp}

\subsection{Product systems}
Suppose $E$ and $F$ are finite-dimensional Hilbert spaces, and $u: F\otimes E\to E\otimes F$ is any unitary operator. If we define $X(0,0)= \mathbb{C}$ and $X(m,n)=E^{\otimes m}\otimes F^{\otimes n}$, then it is easy to see that the coisometries $U^X_{(i,j),(k,l)} = W^u_{(i,j),(k,l)}$ give $X=\{X(m,n)\}_{(m,n)\in \mathbb{N}^2}$ a standard subproduct system structure. Since the coisometries of $X$ are all unitary, we will call it a product system. Moreover, it is clear that any standard subproduct system which is a product system (i.e.~its coisometries are unitary), must be built in this way. Thus, the structure of a product system over $\mathbb{N}^2$ is encoded in a triple $(E,F,u)$. After recalling Lemma \ref{finite-gen}, it is straightforward to verify that if $X$ is a product system with a commutation relation $u$, then $\mathcal{A}_X$ is nothing but $\mathcal{A}_u$ in the sense of \cite{solel-unitary}. Indeed, we will adopt the notation $\mathcal{A}_u$ for product systems, to distinguish them from general tensor algebras of subproduct systems.

Recall that the character space of a unital Banach algebra has an inherent weak-* topology, under which it is compact. Hence, we treat the character space $\mathcal{M}(\mathcal{A}_u)$ as a compact topological space. We will identify it with a certain subset of a Euclidean space, which is an intersection of an algebraic variety with a unit ball of a certain norm. Since we will keep dealing with such objects later on, it will be convenient to introduce the following notation.
\begin{defi}
Given a set of complex polynomials $I\subset \mathbb{C}[z_1,\ldots,z_m,w_1,\ldots, w_n]$, we denote the following subset of $\mathbb{C}^m\times \mathbb{C}^n$:
$$\Omega^{m,n}(I):= \left\{(z,w)\in \mathbb{C}^m\times \mathbb{C}^n\::\: p(z,w)=0,\,\forall p\in I,\:\|z\|,\|w\|\leq 1\right\}$$
where $\|\cdot\|$ is the Euclidean norm. We will call $\Omega^{m,n}(I)$ the \textit{polyball variety} of $I$.
\end{defi}
\textit{Remark:} Of course, like in the setting of usual algebraic varieties, if $\langle I \rangle$ is the ideal in the ring of polynomials which is generated by $I$, then $\Omega^{m,n}(I)= \Omega^{m,n}(\langle I\rangle)$.\\ \\
What follows is a digest of \cite[Proposition 3.1]{solel-unitary}. Suppose $X=\{E^{\otimes i}\otimes F^{\otimes j}\}_{(i,j)\in \mathbb{N}^2}$ is a product system given by the unitary $u: F\otimes E\to E\otimes F$. Let $\{e_i\}_{i=1}^m\subset E$ and $\{f_j\}_{j=1}^n\subset F$ be orthonormal bases, and let $\left(u_{(i,j),(k,l)}\right)\in M_{mn}(\mathbb{C})$ be the representing matrix of $u$ relative to the bases $\{f_l \otimes e_k\}$ and $\{e_i \otimes f_j\}$. For all $1\leq i\leq m$ and $1\leq j\leq n$ we define the polynomial
$$p_{i,j}(z_1,\ldots, z_m, w_1,\ldots, w_n) = w_jz_i- \sum_{k=1}^m\sum_{l=1}^n u_{(k,l),(i,j)}z_kw_l $$
\begin{prop}\label{prod-char}
With the above notations: \\
\textbf{(a)} There exists a homeomorphism $\Psi: \mathcal{M}(\mathcal{A}_u)\to \Omega^{m,n}\left(\{p_{i,j}\}_{i,j=1}^{m,n}\right)$ which is given by
$$\Psi(\alpha) = \left(\alpha\left(L^{(1,0)}_{e_1}\right),\ldots, \alpha\left(L^{(1,0)}_{e_m}\right), \alpha\left( L^{(0,1)}_{f_1}\right),\ldots,\alpha\left(L^{(0,1)}_{f_n}\right)\right)$$
where $\left\{L^{(1,0)}_{e_i}, L^{(0,1)}_{f_j}\right\}\in \mathcal{A}_u$ are the generators of the tensor algebra. \\
\textbf{(b)} If $\Psi(\alpha) = (z,w)\in \mathbb{C}^m\times \mathbb{C}^n$ and $\|z\|,\|w\|<1$, then there exists a vector $w_\alpha\in \mathcal{F}(E,F)$ such that $\alpha(T) = \left\langle T \frac{w_\alpha}{\|w_\alpha\|},\frac{w_\alpha}{\|w_\alpha\|}\right\rangle$ for all $T\in\mathcal{A}_u$. Also, for all $x\in E^{\otimes i}\otimes F^{\otimes j}$, we have $\alpha\left(L^{(i,j)}_x\right)=\langle x, w_\alpha\rangle$.

\end{prop}

\subsection{Fourier coefficients}
The additional tool we will need is the ability to decompose operators in a tensor algebra of a subproduct system via their Fourier coefficients. This tool was applied in various frameworks of operator algebras.  What we are about to define here is very similar to well known constructions. Therefore, we feel confident to skip the bulk of the proofs,
and to depict only the results we need. For more reference, we will mention only \cite{solel-unitary}, in which Fourier coefficients were defined for product systems tensor algebras $\mathcal{A}_u$, and \cite{solel-orr}, in which the same theory was meticulously developed for subproduct systems over $\mathbb{N}$.

Suppose $X$ is a subproduct system, and let $p_{(i,j)}\in B(\mathcal{F}_X)$ be the projection onto $X(i,j)$. Then, $U_{s,t}:= \sum_{k,l=0}^\infty e^{i(ks+lt)}p_{(k,l)}$ defines a two-parameter unitary group in $B(\mathcal{F}_X)$. Next, we define for all $k\geq0$, and $T\in\mathcal{A}_X$,
$$\widetilde{\Phi}_k(T): = \frac1{2\pi}\int_0^{2\pi} e^{-ikt}U_{t,t} TU_{t,t}^\ast \,dt$$ 
It can be checked that in this manner each operator $T\in\mathcal{A}_X$ defines a sequence of operators $\left\{\widetilde{\Phi}_k(T)\right\}_{k=0}^\infty\subset \mathcal{A}_X$ such that for all $k$, $\widetilde{\Phi}_k(T) = \sum_{i=0}^k L^{(i,k-i)}_{\eta_{i,k-i}}$ for some $\left\{\eta_{i,k-i}\in X(i,k-i)\right\}_{i=0}^k$. This sequence can be seen as the components of $T$, because the series $\sum\widetilde{\Phi}_k(T)$ Ces\`{a}ro-converges to $T$. That is, we have the norm limit
$$\lim_{p\to\infty}\sum_{k\leq p}\left(1-\frac{k}p\right)\widetilde{\Phi}_k(T)=T$$
We note that as operators on $\mathcal{A}_X$, $\widetilde{\Phi}_k$ are all linear and completely contractive.

For our purposes, it will also be convenient to work with another set of Fourier coefficients. For each $T\in\mathcal{A}_X$, $k,l\geq0$ we define
$$\Phi_{k,l}(T):= \frac1{4\pi^2}\iint_{[0,2\pi]\times[0,2\pi]} e^{-i(ks+lt)}U_{s,t} TU_{s,t}^\ast \,ds\,dt$$
We then get that $\Phi_{k,l}(T)= L^{(k,l)}_{\scriptstyle{\eta_{\scriptstyle{k,l}}}}$, where $\{\eta_{k,l}\}$ are the vectors described above, and so, $\widetilde{\Phi}_k(T) = \sum_{i=0}^k \Phi_{i,k-i}(T)$. Thus, we still retain the property that $T\in \overline{span}\{\Phi_{i,j}(T)\::\: i,j\geq0\}$. It is also possible to show that for all $T\in\mathcal{A}_X$, and all $i,j\geq0$, we have the weak convergence
$$\sum_{k,l=0}^\infty p_{(k+i, l+j)}Tp_{(k,l)} = \Phi_{i,j}(T)$$

\subsection{The character space}
Suppose $X$ is a standard subproduct system, $E=X(1,0)$, $F=X(0,1)$, and $u$ is a commutation relation for $X$. Then, $\mathcal{F}_X$ is a subspace of $\mathcal{F}(E,F)$. First, we want to show that $\mathcal{A}_X$ has a tight connection to $\mathcal{A}_u$. Namely, the compression of $\mathcal{A}_u\subset B(\mathcal{F}(E,F))$ onto the subspace $\mathcal{F}_X$ serves as a norm-continuous homomorphism from $\mathcal{A}_u$ into $\mathcal{A}_X$, whose image is dense. This connection will allow us to deduce the description of $\mathcal{M}(\mathcal{A}_X)$ from that of $\mathcal{M}(\mathcal{A}_u)$.

\begin{lem}
The subspace $\mathcal{F}_X$ is co-invariant for $\mathcal{A}_u$.
\end{lem}
\begin{proof}
Denote $p_{(m,n)}$ the orthogonal projection from $E^{\otimes m}\otimes F^{\otimes n}$ to $X(m,n)$. 
Take $x\in (E^{\otimes m}\otimes F^{\otimes n})\ominus X(m,n)$. By Lemma \ref{finite-gen}, $\mathcal{A}_u$ is generated by the identity operator and operators of the form $\{L^{(1,0)}_e, L^{(0,1)}_f\}$. Therefore, it is sufficient to show that $L^{(1,0)}_e(x)\perp X(m+1,n)$ and $L^{(0,1)}_f(x)\perp X(m,n+1)$, for all $e\in E, f\in F$. Indeed, since $p_{(m,n)}x = 0$, we have $ (I_E\otimes p_{(m,n)})(e\otimes x) =  (I_F\otimes p_{(m,n)})(f\otimes x) = 0$. So, by the inequalities in Proposition \ref{standard-2} we must have 
$$p_{(m+1,n)}L^{(1,0)}_e(x) = p_{(m+1,n)}W^u_{(1,0),(m,n)}(e\otimes x) = 0$$
$$p_{(m,n+1)}L^{(0,1)}_f(x) = p_{(m,n+1)}W^u_{(0,1),(m,n)}(f\otimes x) = 0$$

\end{proof}
The above lemma shows that the compression mapping $\theta: \mathcal{A}_u\to B(\mathcal{F}_X)$ of operators in $B(\mathcal{F}(E,F))$ onto the subspace $\mathcal{F}_X\subset \mathcal{F}(E,F)$, is a homomorphism. 
Recall, that for all $x\in E^{\otimes i}\otimes F^{\otimes j}$, and $y\in E^{\otimes k}\otimes F^{\otimes l} $, we have $L^{(i,j)}_x(y) = W^u_{(i,j),(k,l)}(x\otimes y)$. So, it can be easily seen from the definition of a standard subproduct system, that $\theta\left(L^{(i,j)}_x\right)= \widehat{L}^{(i,j)}_x$, where $\widehat{L}^{(i,j)}_x$ will be the ad-hoc notation for creation operators in $\mathcal{A}_X$. In other words, the generators of $\mathcal{A}_u$ as a norm-closed algebra are mapped by $\theta$ to the generators of $\mathcal{A}_X$. Therefore, we can conclude that $\theta(\mathcal{A}_u)\subset \mathcal{A}_X$ and it is norm-dense there.

\begin{prop}\label{ker-compress}
Suppose $X$ is a standard subproduct system, with $E=X(1,0)$, $F=X(0,1)$ and $u$ is a commutation relation for $X$. Suppose $\theta:\mathcal{A}_u\to \mathcal{A}_X$ is the compression homomorphism described above. Then, the ideal $\ker \theta \subset \mathcal{A}_u$ is the norm-closure of 
$$span\bigcup_{i,j\geq 0}\left\{L^{(i,j)}_x\::\: x\in E^{\otimes i}\otimes F^{\otimes j}\ominus X(i,j)\right\}$$

\end{prop}
\begin{proof}
For all $i,j\geq0$, denote $K_{i,j}= \ker\theta\cap \left\{L^{(i,j)}_x\::\: x\in E^{\otimes i}\otimes F^{\otimes j}\right\}$. First, we will prove that $\ker \theta$ is the norm-closure of $span\bigcup_{i,j\geq0} K_{i,j}$.
Suppose $T\in \ker \theta$ is given. Then, $\theta\left(U_{s,t}TU_{s,t}^\ast\right)=0$ (when $\theta$ is seen as a function on the whole $B(\mathcal{F}(E,F))$), and so, we see from its defining integral that for all $i,j\geq0$, $\Phi_{i,j}(T)$ is in $\ker \theta$. Thus, we clearly get that $\Phi_{i,j}(T)\in K_{i,j}$. But, since $T$ is in the closed span of its Fourier coefficients, we get that $T$ is in the closure of $span\bigcup_{i,j\geq0} K_{i,j}$.

Now, we need to prove that $K_{i,j} = \left\{L^{(i,j)}_x\::\: x\in E^{\otimes i}\otimes F^{\otimes j}\ominus X(i,j)\right\}$. Suppose $L^{(i,j)}_x\in K_{i,j}$. Then, since $L^{(i,j)}_x\in \ker \theta$, for the unit vector $\Delta= 1\in\mathbb{C}= E^{\otimes 0}\otimes F^{\otimes 0}\subset \mathcal{F}_X$, we get $x = L^{(i,j)}_x(\Delta)\in \mathcal{F}_X^\perp$. Hence, $x\in E^{\otimes i}\otimes F^{\otimes j}\ominus X(i,j)$.

For the reverse inclusion, suppose $x\in E^{\otimes i}\otimes F^{\otimes j}\ominus X(i,j)$. Then, for all $y\in E^{\otimes k}\otimes F^{\otimes l}$, $(p_{(i,j)}\otimes I)(x\otimes y)=0$, where $p_{(i,j)}\in B(E^{\otimes i}\otimes F^{\otimes j})$ is the projection on $X(i,j)$. So, by the inequalities in Proposition \ref{standard-2}, we have
$$L^{(i,j)}_x(y) = W^u_{(i,j),(k,l)}(x\otimes y) \in E^{\otimes i+k}\otimes F^{\otimes j+l}\ominus X(i+j,k+l)\subset \mathcal{F}_X^\perp$$
Hence, $L^{(i,j)}_x\in \ker\theta$.

\end{proof}

Suppose $\dim E = m$ and $\dim F = n$. We are about to show that $\mathcal{M}(\mathcal{A}_X)$ is homeomorphic to a certain polyball variety in $\mathbb{C}^m\times \mathbb{C}^n$. Moreover, this polyball variety will have a certain homogeneity property, for which we will need the following concept.

Given a vector $\overline{z}=(z_1,\ldots,z_p)\in\mathbb{C}^p$, and a multi-index $k = (k_1,\ldots, k_p)\in\mathbb{N}^p$, we adopt the notation $\overline{z}^k= z_1^{k_1}\cdots z_p^{k_p}$ for the corresponding monomial. Also, we denote $|k|=\sum_{i=1}^p k_p$. 
\begin{defi}
A complex polynomial $p\in \mathbb{C}[z_1,\ldots, z_m,w_1,\ldots, w_n]$ will be called \textit{$(m,n)$-homogeneous} of degree $(d_m,d_n)$, if it is of the form 
$$p(\overline{z},\overline{w}) = p(z_1,\ldots,w_n) = \sum_{|k|=d_m,\,|l|=d_n} a_{k,l}\overline{z}^k\overline{w}^l$$
An ideal $J\subset \mathbb{C}[z_1,\ldots, z_m,w_1,\ldots, w_n]$ will be called \textit{$(m,n)$-homogeneous}, if it is generated by $(m,n)$-homogeneous polynomials. A polyball variety $\Omega^{m,n}(J)$ we be called \textit{homogeneous}, if $J$ is $(m,n)$-homogeneous.
\end{defi}

Given fixed orthonormal bases $\{e_i\}_{i=1}^m$ of $E$, and $\{f_j\}_{j=1}^n$ of $F$, we assign to each vector $x\in E^{\otimes i}\otimes F^{\otimes j}$ a complex $(m,n)$-homogeneous polynomial of degree $(i,j)$, in a natural manner: Given 
$$x = \sum_{s_1,\ldots,s_i=1}^m\sum_{t_1,\ldots,t_j=1}^n a_{s_1,\ldots,s_i}^{t_1,\ldots, t_j} e_{s_1}\otimes\cdots\otimes e_{s_i}\otimes f_{t_1}\otimes\cdots\otimes f_{t_j}$$
let $q^x$ be the polynomial defined by the formula
$$q^x(z_1,\ldots,z_m, w_1,\ldots,w_n) = \sum_{s_1,\ldots,s_i=1}^m\sum_{t_1,\ldots,t_j=1}^n a_{s_1,\ldots,s_i}^{t_1,\ldots, t_j} z_{s_1}\cdots z_{s_i}w_{t_1}\cdots w_{t_j}$$

\textit{Remark:} The transformation $x\mapsto q^x$ implements a passage from a non-commutative domain to a commutative one. Hence, it may occur that a non-zero vector $x$ will yield the zero polynomial $q^x\equiv 0$. For example, if $x= e_1\otimes e_2 - e_2\otimes e_1 \not=0$, then $q^x = z_1z_2-z_2z_1\equiv 0$.

Finally, we are ready to prove the principal result of this section.

\begin{thm}\label{chars}
Suppose $X$ is a standard subproduct system, with $E=X(1,0)$, $F=X(0,1)$, $\dim E = m$, $\dim F = n$, and $u$ is a commutation relation for $X$. Then, there exists a $(m,n)$-homogeneous ideal of polynomials $J$, such that the character space $\mathcal{M}(\mathcal{A}_X)$ (equipped with the weak-* topology) is homeomorphic to the homogeneous polyball variety $\Omega^{m,n}(J)$.\\
More specifically, if $\{p_{i,j}\}_{i,j=1}^{m,n}$ are the polynomials associated with $u$, relative to some choice of orthonormal bases for $E$ and $F$, then $J$ is the ideal generated by the set
$$\{p_{i,j}\}_{i,j=1}^{m,n}\cup \bigcup_{i,j\geq0}\left\{q^x\::\: x\in E^{\otimes i}\otimes F^{\otimes j}\ominus X(i,j)\right\}$$
where $q^x$ is the polynomial associated with the vector $x$, relative to the same choice of bases.
\end{thm}

\begin{proof}
Let $\{e_i\}_{i=1}^n$, $\{f_j\}_{j=1}^m$ be a choice orthonormal bases for $E$, $F$, respectively, and denote $J$ to be the ideal defined in the statement, relative to this choice. 

Recall, (Proposition \ref{prod-char}) that $\mathcal{M}(\mathcal{A}_u)$ is homeomorphic to $\Omega^{m,n}(\{p_{i,j}\})$. We will denote $\alpha_{(z,w)}\in\mathcal{M}(\mathcal{A}_X)$ to be the character mapped to the vector $(z,w)\in\mathbb{C}^m\times \mathbb{C}^n$ under this homeomorphism. Recall, that $\alpha_{(z,w)}\left(L^{(1,0)}_{e_i}\right)=z_i$ and $\alpha_{(z,w)}\left(L^{(0,1)}_{f_j}\right)=w_j$, where $(z,w) = (z_1,\ldots,z_m,w_1,\ldots,w_n)$. From that, it is easy to obtain, that for all $x\in E^{\otimes i}\otimes F^{\otimes j}$, 
$$\alpha_{(z,w)}\left(L^{(i,j)}_x\right) = q^x(z,w)$$
As seen before, we have the (continuous) compression homomorphism $\theta:\mathcal{A}_u\to\mathcal{A}_X$. It induces a continuous mapping $\theta^\ast:\mathcal{M}(\mathcal{A}_X)\to \mathcal{M}(\mathcal{A}_u)$, given by $\theta^\ast(\alpha) = \alpha\circ\theta$ on characters $\alpha$. Since the image of $\theta$ is dense inside $\mathcal{A}_X$, each character on $\mathcal{A}_X$ is defined by its values on $\theta(\mathcal{A}_u)$. (That is because all characters on a Banach algebra are continuous.) Hence, $\theta^\ast$ is injective.

Suppose $\theta^\ast(\alpha) = \alpha_{(z,w)}$ for some $\alpha\in\mathcal{A}_X$, and suppose $x\in E^{\otimes i}\otimes F^{\otimes j}\ominus X(i,j)$ for some $i,j$. By Proposition \ref{ker-compress}, $L^{(i,j)}_x\in \ker \theta$, hence, $q^x(z,w)=\alpha_{(z,w)}\left(L^{(i,j)}_x\right)=0$. Thus, under the identification $\mathcal{M}(\mathcal{A}_u)\cong \Omega^{m,n}(\{p_{i,j}\})$, we have,
$$\theta^\ast(\mathcal{M}(\mathcal{A}_X))\subset \Omega^{m,n}(J)\subset \Omega^{m,n}(\{p_{i,j}\})$$
We are left to show that the left inclusion above is an equality. Suppose $(z,w)\in \Omega^{m,n}(J)$ such that $\|z\|,\|w\|<1$. Then, by Proposition \ref{prod-char}(b), there is a vector $w_{(z,w)}\in \mathcal{F}(E,F)$ such that $\alpha_{(z,w)}(T) = \left\langle T \frac{w_{(z,w)}}{\|w_{(z,w)}\|},\frac{w_{(z,w)}}{\|w_{(z,w)}\|}\right\rangle$ for all $T\in\mathcal{A}_u$. But, also by the same proposition, for all $x\in E^{\otimes i}\otimes F^{\otimes j}\ominus X(i,j)$, we have,
We are left to show that the left inclusion above is an equality. Suppose $(z,w)\in \Omega^{m,n}(J)$ such that $\|z\|,\|w\|<1$. Then, by Proposition \ref{prod-char}(b), there is a vector $w_{(z,w)}\in \mathcal{F}(E,F)$ such that $\alpha_{(z,w)}(T) = \left\langle T \frac{w_{(z,w)}}{\|w_{(z,w)}\|},\frac{w_{(z,w)}}{\|w_{(z,w)}\|}\right\rangle$ for all $T\in\mathcal{A}_u$. But, also by the same proposition, for all $x\in E^{\otimes i}\otimes F^{\otimes j}\ominus X(i,j)$, we have,
$$q^x\in J\quad \Rightarrow\quad \langle x, w_{(z,w)}\rangle = \alpha_{(z,w)}\left(L^{(i,j)}_x\right)=q^x(z,w) = 0$$
Thus, we must have $w_{(z,w)}\in \mathcal{F}_X$, and that means the character $\alpha_{(z,w)}$ factors through the compression $\theta$ of $\mathcal{A}_u$ onto $\mathcal{F}_X$. But, $\alpha(T):=\left\langle T \frac{w_{(z,w)}}{\|w_{(z,w)}\|},\frac{w_{(z,w)}}{\|w_{(z,w)}\|}\right\rangle$, which has just been seen to define a character on the image of $\theta$, obviously can be extended to a character on its closure $\mathcal{A}_X$. In other words, $\alpha\in \mathcal{M}(\mathcal{A}_X)$ and $\theta^\ast(\alpha) = \alpha_{(z,w)}$.

At this point, we know $\theta^\ast$ is a continuous injective mapping from $\mathcal{M}(\mathcal{A}_X)$ into the compact space $\Omega^{m,n}(J)$, and its image contains a dense subset of (namely, $\Omega^{m,n}(J)\cap (\mathbb{B}^m\times\mathbb{B}^n)$). Since $\mathcal{M}(\mathcal{A}_X)$ is compact, that must mean $\theta^\ast$ is a homeomorphism onto $\Omega^{m,n}(J)$.

\end{proof}

\section{Ideals of non-commutative polynomials}
We present another perspective on subproduct systems of finite-dimensional Hilbert spaces over $\mathbb{N}\times \mathbb{N}$, with a more algebraic flavor. Given a product system $(E,F,u)$, we would like to find all standard subproduct systems that can be embedded inside it. It turns out that one can find a parameterization of this collection of embedded subproduct systems, when looking at the algebra of complex non-commutative polynomials (with the right number of variables). In \cite{solel-orr}, such a parametrization was made for subproduct systems over $\mathbb{N}$, simply by coupling subproduct systems with ideals in the algebra of complex non-commutative polynomials. Motivated by this identification, we develop a similar tool for our setting.

Suppose $E$ and $F$ are finite dimensional Hilbert spaces, and $u:F\otimes E\to E\otimes F$ is a unitary operator. We would like to consider the \textit{algebraic} Fock space, which is the direct sum of the inner product spaces (without completion): $$\mathcal{F}^{alg}(E,F):= \bigoplus_{(m,n)\in \mathbb{N}^2}E^{\otimes m}\otimes F^{\otimes n}$$ 
The vector space $\mathcal{F}^{alg}(E,F)$, unlike the usual Fock space, can be turned into an algebra by linearly extending the formula:
$$x\cdot y:= W^u_{(i,j),(k,l)}(x\otimes y)\quad \forall\: x\in E^{\otimes i}\otimes F^{\otimes j}\,\;y\in E^{\otimes k}\otimes F^{\otimes l}$$  
By applying the remark after Proposition \ref{standard-2} on the case of a product system, one can see this indeed defines an associative multiplication.

We will call a subspace $\mathcal{M}\subset \mathcal{F}^{alg}(E,F)$ homogeneous, if $\mathcal{M}=\bigoplus_{(m,n)\in \mathbb{N}^2} \mathcal{M}_{(m,n)}$ for some subspaces $\mathcal{M}_{(m,n)}\subset E^{\otimes m}\otimes F^{\otimes n}$.

Given a standard subproduct system $X$ with $X(1,0)\subset E$ and $X(0,1)\subset F$, we can view the algebraic direct sum $\mathcal{F}^{alg}_X:= \bigoplus_{(m,n)\in \mathbb{N}^2} X(m,n)$ as a homogeneous subspace of $\mathcal{F}^{alg}(E,F)$. We will also refer to the orthogonal complement 
$$\mathcal{F}^{alg\perp}_X= \bigoplus_{(m,n)\in \mathbb{N}^2} X(m,n)^\perp\subset \mathcal{F}^{alg}(E,F)$$
where $X(m,n)^\perp= E^{\otimes m}\otimes F^{\otimes n}\ominus X(m,n)$. It turns out that in these terms, standard subproduct systems get another characterization.
\begin{prop}\label{subprod-ideals} 
A homogeneous subspace $\mathcal{M}\subset \mathcal{F}^{alg}(E,F)$ is a proper ideal, if and only if, there exists a standard subproduct system $X$ with $X(1,0)\subset E$, $X(0,1)\subset F$ such that $\mathcal{M}= \mathcal{F}^{alg \perp}_X$, and $X$ has $u$ as a commutation relation.

If exists, such a subproduct system $X$ is unique.
\end{prop}
\begin{proof}
Suppose $\mathcal{M}=\bigoplus_{(m,n)\in \mathbb{N}^2} \mathcal{M}_{(m,n)}$ is a proper homogeneous ideal in the algebra $\mathcal{F}^{alg}(E,F)$. This means that for all $x\in \mathcal{M}_{(i,j)}$ and $y\in E^{\otimes k}\otimes F^{\otimes l}$, we have
$$W^u_{(i,j),(k,l)}(x\otimes y), W^u_{(k,l),(i,j)}(y\otimes x)\in \mathcal{M}_{(i+k,j+l)}$$
So, since $W^u_{(i,j),(k,l)}$ are unitary we have, 
$$\mathcal{M}_{(i+k,j+l)}^\perp\subset W^u_{(i,j),(k,l)}\left( \mathcal{M}^\perp_{(i,j)}\otimes E^{\otimes k}\otimes F^{\otimes l}\right)$$ $$\mathcal{M}_{(i+k,j+l)}^\perp\subset W^u_{(k,l),(i,j)}\left( E^{\otimes k}\otimes F^{\otimes l}\otimes \mathcal{M}^\perp_{(i,j)}\right)$$ where $\mathcal{M}^\perp_{(m,n)}= E^{\otimes m}\otimes F^{\otimes n}\ominus \mathcal{M}_{(m,n)}$. When interchanging $i\leftrightarrow k$ and $j\leftrightarrow l$ in the second inclusion above, we get that $\mathcal{M}_{(i+k,j+l)}^\perp\subset W^u_{(i,j),(k,l)}\left( \mathcal{M}^\perp_{(i,j)}\otimes \mathcal{M}^\perp_{(k,l)}\right)$. Also, $\mathcal{M}_{(0,0)}=\{0\}\;\Rightarrow\; \mathcal{M}_{(0,0)}^\perp=\mathbb{C}$ because $\mathcal{M}$ is proper. This means that $X(m,n):=\mathcal{M}^\perp_{(m,n)}$ defines a standard subproduct system with a commutation relation $u$, and $\mathcal{M}= \mathcal{F}^{alg \perp}_X$.

Conversely, given a standard subproduct system $X$ as such, by the inequalities in Proposition \ref{standard-2}, we must have
$$(E^{\otimes k}\otimes F^{\otimes l})\cdot X(i,j)^\perp,\quad X(i,j)^\perp \cdot (E^{\otimes k}\otimes F^{\otimes l}) \subset X(i+k,j+l)^\perp$$
for all $(i,j),(k,l)\in \mathbb{N}^2$, and $X(0,0)^\perp=\{0\}$. Thus, $\mathcal{F}^{alg\perp}_X = \bigoplus_{(m,n)\in\mathbb{N}^2} X(m,n)^\perp$ is a proper ideal in the algebra $\mathcal{F}^{alg}(E,F)$.

Uniqueness of $X$ is obvious, since $\mathcal{M}$ is homogeneous and $X(m,n)=\mathcal{M}^\perp_{(m,n)}$.

\end{proof}

The last proposition gives us an inclusion-reversing correspondence between the collection of standard subproduct systems related with the triple $(E,F,u)$ and the collection of ideals in a certain non-commutative finitely generated complex algebra. The next step will be to develop this correspondence further, by identifying these ideals with ideals of \textit{non-commutative polynomials} of a certain form.

To do that, we need to fix a choice of orthonormal bases $\{e_1,\ldots,e_m\}\subset E$ and  $\{f_1,\ldots, f_n\}\subset F$. Let $\mathcal{P}_{m,n} =\mathbb{C}\langle z_1,\ldots, z_m, w_1,\ldots, w_n\rangle$ denote the algebra of non-commutative complex polynomials on $m+n$ variables. We construct a unital linear homomorphism $\Phi: \mathcal{P}_{m,n}\to \mathcal{F}^{alg}(E,F)$ by defining $\Phi(z_i)= e_i\in E\subset \mathcal{F}^{alg}(E,F)$ for all $1\leq i\leq m$, and similarly $\Phi(w_j)=f_j\in F$ for all $1\leq j\leq n$. Since $\mathcal{P}_{m,n}$ is a free algebra over the generators $\{z_i,w_j\}$, $\Phi$ is well-defined.

\begin{lem}\label{hard-lemma}
Let $\left(u_{(i,j),(k,l)}\right)\in M_{mn}(\mathbb{C})$ be the representing matrix of $u$ relative to the bases $\{f_l \otimes e_k\}$ and $\{e_i \otimes f_j\}$. For all $1\leq i\leq m$ and $1\leq j\leq n$ we define the non-commutative polynomials
$$P_{i,j}= w_jz_i- \sum_{k=1}^m\sum_{l=1}^n u_{(k,l),(i,j)}z_kw_l \in \mathcal{P}_{m,n}$$
Then, $\Phi:\mathcal{P}_{m,n}\to \mathcal{F}^{alg}(E,F)$ is a surjective homomorphism whose kernel is the ideal generated by the set $\{P_{i,j}\}_{i,j=1}^{m,n}$.
\end{lem}

\begin{proof}
Let $I\subset \mathcal{P}_{m,n}$ be the ideal generated by $\{P_{i,j}\}$. Let $\Psi: \mathcal{F}^{alg}(E,F)\to \mathcal{P}_{m,n}$ be the linear map defined by 
$$\Psi\left(e_{i_1}\otimes \cdots\otimes e_{i_s}\otimes f_{j_1}\otimes \cdots \otimes f_{j_t}\right) = z_{i_1}\cdots z_{i_s}w_{j_1}\cdots w_{j_t}$$
Clearly, $\Phi\circ \Psi$ is the identity map, and that shows $\Phi$ is surjective. Note that,
$$\Psi\circ\Phi(w_jz_i) = \Psi(u(f_j\otimes e_i)) = \sum_{k=1}^m\sum_{l=1}^n u_{(k,l),(i,j)}z_kw_l= w_jz_i-P_{i,j}$$
Note further, that the multiplication in $\mathcal{F}^{alg}(E,F)$ is essentially a successive application of $u$ on parts of the tensor product of the multiplicands. Therefore, it is easy to proceed with the above calculation to verify that for all $Q\in \mathcal{P}_{m,n}$, the difference between $Q$ and $\Psi\circ \Phi(Q)$ can be written as a sum of multiples of elements of the form $w_jz_i- \Psi\circ\Phi(w_jz_i)= P_{i,j}$. In other words, $Q-\Psi\circ \Phi(Q)\in I$, and when $\Phi(Q)=0$, we have $Q \in I$. Hence, $\ker \Phi\subset I$, and conversely, 
$$\Phi(P_{i,j})=\Phi(w_jz_i)- \Phi\circ\Psi\circ\Phi(w_jz_i) = \Phi(w_jz_i)-\Phi(w_jz_i)=0\quad \forall i,j$$
$$\Rightarrow\quad I\subset \ker\Phi$$
\end{proof}
We see that our algebra of interest in this section, $\mathcal{F}^{alg}(E,F)$, can be identified as a well understood quotient of $\mathcal{P}_{m,n}$. Thus, we get a natural correspondence between ideals in $\mathcal{F}^{alg}(E,F)$, and ideals of non-commutative polynomials. So, in light of Proposition \ref{subprod-ideals}, to complete this viewpoint on subproduct systems we need to describe the ideals in $\mathcal{P}_{m,n}$ that are mapped under $\Phi$ to homogeneous subspaces in $\mathcal{F}^{alg}(E,F)$.

Let $\epsilon:\mathcal{P}_{m,n}\to \mathbb{C}[z_1,\ldots,z_m,w_1,\ldots, w_n]$ be the natural mapping that takes a non-commutative polynomial into its commutative version.
\begin{defi}
A non-commutative complex polynomial $P\in\mathcal{P}_{m,n}$ 
will be called \textit{$(m,n)$-homogeneous} of degree $(d_m,d_n)$, if $P= \sum_i Q_i$, where $\{Q_i\}$ are monomials such that for all $i$, 
$\epsilon(Q_i)$ is an $(m,n)$-homogeneous (commutative) polynomial of degree $(d_m,d_n)$.\\
An ideal $I\subset \mathcal{P}_{m,n}$ will be called \textit{$(m,n)$-homogeneous}, if it is generated by $(m,n)$-homogeneous non-commutative polynomials.
\end{defi}
For example, $P(z,w)= wzw+z^2w$ is a non-commutative polynomial, that is \textit{not} $(1,1)$-homogeneous, while $Q(z_1,z_2,w)= wz_1wz_2+ z_1^2w^2$ is $(2,1)$-homogeneous.

\begin{thm}\label{non-comm-ideal}
Suppose $E,F$ are two Hilbert spaces with fixed finite orthonormal bases $\{e_1,\ldots,e_m\}\subset E$ and $\{f_1,\ldots, f_n\}\subset F$, and $u: F\otimes E\to E\otimes F$ is a unitary operator. Let $\{P_{i,j}\}_{i,j=1}^{m,n}\subset \mathcal{P}_{m,n}$ be the non-commutative polynomials as defined in Lemma \ref{hard-lemma} relative to $u$ and the chosen bases.

Then, there is an inclusion-reversing bijection between the following collections:
\begin{itemize}
\item Standard subproduct systems $X$ over $\mathbb{N}^2$ that have $u$ as a commutation relation, and that $X(1,0)\subset E$, $X(0,1)\subset F$.
\item Proper $(m,n)$-homogeneous ideals in $\mathcal{P}_{m,n}$ that contain all of $\{P_{i,j}\}$.
\end{itemize}
\end{thm}
\begin{proof}
By Proposition \ref{subprod-ideals} there is an inclusion-reversing bijection between the collection of subproduct systems described in the statement, and the collection of proper homogeneous ideals in the algebra $\mathcal{F}^{alg}(E,F)$. But, by Lemma \ref{hard-lemma}, the mapping $\Phi$ gives an inclusion-preserving bijection between ideals in $\mathcal{F}^{alg}(E,F)$, and ideals in $\mathcal{P}_{m,n}$ that contain all of $\{P_{i,j}\}$. Thus, we are left to show that an ideal $I\subset \mathcal{F}^{alg}(E,F)$ is a homogeneous subspace, if and only if, $\Phi^{-1}(I)$ is an $(m,n)$-homogeneous ideal. Indeed, $I$ is homogeneous, if and only if, it is generated by vectors in $\bigcup_{(i,j)\in \mathbb{N}^2} E^{\otimes i}\otimes F^{\otimes j}$. That, in turn, is equivalent to $\Phi^{-1}(I)$ being generated by a subset of $\bigcup_{(i,j)\in \mathbb{N}^2} \Phi^{-1}(E^{\otimes i}\otimes F^{\otimes j})$. But, from the definition of $\Phi$, this is exactly the collection of $(m,n)$-homogeneous non-commutative polynomials in $\mathcal{P}_{m,n}$.
\end{proof}
The above theorem supplies the existence of an abundance of examples for standard subproduct systems. One needs only to point at a set of $(m,n)$-homogeneous non-commutative polynomials that will generate (together with $\{P_{i,j}\}$) an ideal in $\mathcal{P}_{m,n}$. In fact, it is possible to use this theorem to supply an alternative proof to Proposition \ref{construct}, by constructing the appropriate ideal instead of constructing the desired subproduct system directly.

\section{The tensor algebra as a complete invariant}
Suppose $X$, $Y$ are two subproduct systems over $\mathbb{N}^2$. It is fairly clear that when $X$ and $Y$ are isomorphic, their tensor algebras $\mathcal{A}_X$ and $\mathcal{A}_Y$ are isometrically isomorphic. In the following sections we would like to address the converse. That is, in case $\mathcal{A}_X$ is isomorphic to $\mathcal{A}_Y$, what can be said about the underlying subproduct systems $X$ and $Y$? In general, this question seems hard to treat. As we shall see the tensor algebras of subproduct systems possess two natural gradations by both the monoids $\mathbb{N}$ and $\mathbb{N}\times\mathbb{N}$. It will be shown that if the isomorphism between the tensor algebras respects those gradations, then under suitable conditions the underlying subproduct systems must be isomorphic. Therefore, given an isometric isomorphism $\phi:\mathcal{A}_X\to\mathcal{A}_Y$, the issue at hand becomes a search for a sufficient condition which will force $\phi$ to preserve the $\mathbb{N}\times \mathbb{N}$-gradation (or, at least, the $\mathbb{N}$-gradation). 

We use these tools to prove the main result of this section, Theorem \ref{complete-inv}, which does not use the language of gradations, but rather looks at the homeomorphism $\phi^\ast$ between $\mathcal{M}(\mathcal{A}_X)$ and $\mathcal{M}(\mathcal{A}_Y)$, which is induced by $\phi$.

Finally, we try to approach the above question without any assumptions on the nature of the algebra isomorphism. In this general case, we find that $\phi^\ast$ actually preserves some differential properties of the polyball varieties associated with the character spaces. As a consequence, we produce two numerical values for a subproduct system, that must be equal for $X$ and $Y$. \\ \\
But, first, we would like to show that the question raised above is not trivial. That is, there are non-isomorphic subproduct systems whose tensor algebras are isometrically isomorphic. Note, that this comes in contrast to the situation in subproduct systems over $\mathbb{N}$. In that setting, \cite[Theorem 4.8]{orr-subprod} denies the existence of such non-isomorphic subproduct systems.

\begin{examp}
Choose $m,n\geq0$, and define a subproduct system $X^{m,n}=\{X(i,j)\}_{(i,j)\in \mathbb{N}^2}$ by setting $X(i,j)=\{0\}$ if $i+j\geq2$, and $X(1,0)=\mathbb{C}^m$, $X(0,1)=\mathbb{C}^n$, $X(0,0)=\mathbb{C}$. It is clear that there is only one possible subproduct system structure on $X$, after restricting to the mentioned dimensions. Then, $\mathcal{F}_X=\mathbb{C}\oplus \mathbb{C}^m\oplus \mathbb{C}^n$, and we can realize operators in $\mathcal{A}_X$ as complex $(m+n+1)\times(m+n+1)$ matrices. When choosing some orthonormal bases $\{e_i\}_{i=2}^{m+1}\subset X(1,0)$, $\{f_j\}_{j=m+2}^{m+n+1}\subset X(0,1)$ and $\Delta=1\in X(0,0)$, the matrix representations relative to the basis $\{\Delta,e_2,\ldots, e_{m+1},f_{m+2},\ldots, f_{m+n+1}\}$ of $\mathcal{F}_X$ will be of the form
$$L^{(1,0)}_{e_i} = \left(\begin{array}{cccc} \delta_{i1} & 0&\cdots  & 0  \\ \vdots & \vdots & & \vdots \\ \delta_{i,m+n+1}& 0& \cdots & 0\end{array}\right)\quad L^{(0,1)}_{f_j} = \left(\begin{array}{cccc} \delta_{j1} & 0&\cdots  & 0  \\ \vdots & \vdots & & \vdots \\ \delta_{j,m+n+1}& 0& \cdots & 0\end{array}\right)$$
and $L^{(0,0)}_\Delta = I$. Thus, we see clearly that when $m+n=m'+n'$, we have $\mathcal{A}_{X^{m,n}}\cong \mathcal{A}_{X^{m',n'}}$. Yet, for every choice of $m,n,m',n'$ such that $\{m,n\}\not=\{m',n'\}$ it is clear that $X^{m,n}$ is \textit{not} isomorphic to $X^{m',n'}$, because of dimensions mismatch.\\
\end{examp}
 
The strongest property that a subproduct system can possess with regard to the issue at hand deserves a terminology:
\begin{defi}
We say a subproduct system $X$ \textit{uniquely defines its algebra} if for all subproduct systems $Y$ such that $\mathcal{A}_X\cong \mathcal{A}_Y$ (in the sense of an isometric isomorphism), we must have $X\cong Y$.
\end{defi}
Thus, $X^{m,n}$ in the example above are subproduct systems that do not uniquely define their algebra.

\subsection{Graded isomorphisms}
Before presenting the next definition we need to introduce a simple notation: Given a subproduct system $X$, and a vector $\eta= (x,y)\in X(1,0)\oplus X(0,1)$, we denote $L^X_\eta:= L^{(1,0)}_x+L^{(0,1)}_y\in \mathcal{A}_X$. 
\begin{defi}\label{graded}
Let $X$ and $Y$ be subproduct systems over $\mathbb{N}^2$, and $\phi:\mathcal{A}_X\to \mathcal{A}_Y$ an isomorphism. \\
We say that $\phi$ is \textit{$\mathbb{N}$-graded} if for all $\eta\in X(1,0)\oplus X(0,1)$, there exists $A\eta\in Y(1,0)\oplus Y(0,1)$ such that $\phi\left(L^X_\eta\right) = L^Y_{A\eta}$.\\
We say that $\phi$ is \textit{$\mathbb{N}^2$-graded} for all $\eta\in X(1,0)$ and $\xi\in X(0,1)$, there are $B\eta\in Y(\pi(1,0))$ and $C\xi\in Y(\pi(0,1))$ for which
$$\phi\left(L^{(1,0)}_\eta\right) = L^{\pi(1,0)}_{B\eta}\quad\phi\left(L^{(0,1)}_\xi\right) = L^{\pi(0,1)}_{C\xi}$$
where $\pi$ is either the identity or the coordinate switch on $\mathbb{N}^2$. 
\end{defi}
\textit{Remark}: If they exist, the mappings $A, B, C$, that were implicitly defined, must be linear, and since $\phi$ is an isomorphism they are also invertible. Moreover, if $\phi$ is $\mathbb{N}^2$-graded and isometric, then $B, C$ are unitary operators. This is seen from Lemma \ref{norm-eq}, since $\|\eta\|= \|L^{(1,0)}_\eta\| = \|L^{\pi(1,0)}_{B\eta}\| = \|B\eta\|$. \\

The motivation behind those definitions may not be clear at first sight. But, notice, that the tensor algebra contains the following sub-algebra
$$\widetilde{\mathcal{A}}_X:= \bigoplus_{(i,j)\in \mathbb{N}^2} \left\{L^{(i,j)}_x\::\:x\in X(i,j)\right\}\subset \mathcal{A}_X$$
where the direct sum decomposition is algebraic. This sub-algebra is norm-dense in $\mathcal{A}_X$, as a consequence of the discussion about Fourier coefficients. Also, $\widetilde{\mathcal{A}}_X$ is clearly a $\mathbb{N}\times\mathbb{N}$-graded algebra. Moreover, it can also be seen as a $\mathbb{N}$-graded algebra, if we write it as
$$\widetilde{\mathcal{A}}_X = \bigoplus_{n\in \mathbb{N}} \left(\bigoplus_{i=0}^n \left\{L^{(i,n-i)}_x\::\:x\in X(i,n-i)\right\}\right)$$
It is easy to check that an isomorphism $\phi:\mathcal{A}_X\to\mathcal{A}_Y$ is $\mathbb{N}$( or $\mathbb{N}^2$)-graded by our definition, is equivalent to the fact that $\phi|_{\widetilde{\mathcal{A}}_X}$ takes values in $\widetilde{\mathcal{A}}_Y$ and preserves the $\mathbb{N}$( or $\mathbb{N}^2$)-gradation. (In the $\mathbb{N}\times\mathbb{N}$ case, we allow it to preserve the gradation up to a switch of the coordinates.)

\begin{prop}\label{n2-to-isom}
Let $X, Y$ be subproduct systems over $\mathbb{N}^2$. Then, $X$ and $Y$ are isomorphic if and only if there exists an $\mathbb{N}^2$-graded isometric isomorphism $\phi: \mathcal{A}_X\to \mathcal{A}_Y$. Moreover, in this case $\phi$ is unitarily implemented, that is, there exists a unitary $V: \mathcal{F}_X\to\mathcal{F}_Y$ such that $\phi(T) = VTV^\ast$ for all $T\in\mathcal{A}_X$. 
\end{prop}
\begin{proof}
Suppose $X$ and $Y$ are isomorphic. Then, we can build a unitary $V:\mathcal{F}_X\to\mathcal{F}_Y$ in the same manner as in the proof of Corollary \ref{standard}. It is easy to verify that $V$ implements a $\mathbb{N}^2$-graded isometric isomorphism of $\mathcal{A}_X$ to $\mathcal{A}_Y$.

Conversely, suppose $\phi: \mathcal{A}_X\to \mathcal{A}_Y$ a $\mathbb{N}^2$-graded isometric isomorphism. Let $\pi$ be the monoid isomorphism, and $B, C$ the unitaries that exist by the definition of $\mathbb{N}^2$-graded isomorphisms.
We can define the following unitary operator:
$$\widetilde{V} = \oplus_{(i,j)}(B^{\otimes i}\otimes C^{\otimes j}):\:\bigoplus_{i,j=0}^\infty X(1,0)^{\otimes i}\otimes X(0,1)^{\otimes j}\cong \mathcal{F}(X(1,0),X(0,1))\to$$ $$ \to \bigoplus_{i,j=0}^\infty Y(\pi(1,0))^{\otimes i}\otimes Y(\pi(0,1))^{\otimes j}\cong \mathcal{F}(Y(\pi(1,0)),Y(\pi(0,1))) $$
Let $u_X,u_Y$ be commutation relations for the respective subproduct systems. Denote, also, the operator
$$W = \oplus_{(i,j)}W^{u_Y}_{(0,i),(j,0)}: \bigoplus_{i,j=0}^\infty Y(0,1)^{\otimes i}\otimes Y(1,0)^{\otimes j}\to \bigoplus_{i,j=0}^\infty Y(1,0)^{\otimes i}\otimes Y(0,1)^{\otimes j}$$
and, now, we can define
$$V = \left\{\begin{array}{ll} \widetilde{V} & \pi = id\\ W\widetilde{V} & \pi = \mbox{\small{switch}} \end{array}\right.$$
In both cases $V$ will be a unitary operator whose range is $\mathcal{F}(Y(1,0),Y(0,1))$.

Since we are interested only with the isomorphism class of the subproduct systems involved, by Proposition \ref{standard-2} we may assume that $X$ and $Y$ are standard, and that $\mathcal{F}_X$ and $\mathcal{F}_Y$ are subspaces of the domain and the range of $V$, respectively. We want to show that $V|_{\mathcal{F}_X}$ gives a unitary operator into $\mathcal{F}_Y$. For that purpose, it is enough to check that $\|p_{(i,j)}^X\xi\|=\|p_{\pi(i,j)}^YV\xi\|$ for all $\xi\in X(1,0)^{\otimes i}\otimes X(0,1)^{\otimes j}$, where $p_{(i,j)}^X, p_{(i,j)}^Y$ are the projections as in the statement of Proposition \ref{standard-2}.

Suppose $\xi = \sum_{p=1}^t e_{p,1}\otimes \ldots\otimes e_{p,i}\otimes f_{p,i+1}\otimes\ldots \otimes f_{p,i+j}$ where $e_{p,l}\in X(1,0)$ and $f_{p,l}\in X(0,1)$. When going through the arguments of the proof of Lemma \ref{finite-gen}, it is evident that $L^{(i,j)}_{p^X_{(i,j)}\xi} = \sum_{p=1}^t L^{(1,0)}_{e_{p,1}}\cdots L^{(0,1)}_{f_{p,i+j}}$. So, by Lemma \ref{norm-eq},
$$\|p_{(i,j)}^X\xi\| =  \left\|\sum_{p=1}^t L^{(1,0)}_{e_{p,1}}\cdots L^{(0,1)}_{f_{p,i+j}}\right\|= \left\|\phi\left(\sum_{p=1}^t L^{(1,0)}_{e_{p,1}}\cdots L^{(0,1)}_{f_{p,i+j}}\right)\right\| =$$ $$= \left\|\sum_{p=1}^t L^{\pi(1,0)}_{Be_{p,1}}\cdots L^{\pi(0,1)}_{Cf_{p,i+j}}\right\|= \|p_{\pi(i,j)}^YV\xi\|$$
Thus, we can refer to the unitary $V:\mathcal{F}_X\to\mathcal{F}_Y$. 

We need to show that $V$ is a product system isomorphism. For that purpose, we are about to prove that 
$$q(C\otimes B)=r(B\otimes C)u_X$$ 
where
$$q = \left\{\begin{array}{ll} p_{(1,1)}^Yu_Y & \pi = id\\ p_{(1,1)}^Y & \pi = \mbox{\small{switch}} \end{array}\right.\quad r = \left\{\begin{array}{ll} p_{(1,1)}^Y & \pi = id\\ p_{(1,1)}^Yu_Y & \pi = \mbox{\small{switch}} \end{array}\right.$$
Indeed, pick $e\otimes f\in X(1,0)\otimes X(0,1)$, and denote $u_X(f\otimes e) = \sum e_p\otimes f_p$. Then, 
$$\phi(\sum L^{(1,0)}_{e_p}L^{(1,0}_{f_p})\Delta_Y = \sum L^{\pi(1,0)}_{Be_i}L^{\pi(0,1)}_{Cf_i}\Delta_Y = r(B\otimes C)(\sum e_i\otimes f_i) = r(B\otimes C)u_X(f\otimes e)$$
On the other hand, we have $\sum L^{(1,0)}_{e_p}L^{(0,1)}_{f_p}= L^{(0,1)}_fL^{(1,0)}_e$, and that means,
$$\phi(\sum L^{(1,0)}_{e_p}L^{(0,1)}_{f_p})\Delta_Y = L^{\pi(0,1)}_{Cf}L^{\pi(1,0)}_{Be}\Delta_Y = q(C\otimes B)(f\otimes e)$$
So, we can claim, that for all $\eta\in X(i,j)$ and $\xi\in X(k,l)$ we have 
$$p_{\pi(i+k,j+l)}^YW^{u_Y}_{\pi(i,j),\pi(k,l)}(V\eta \otimes V\xi) = p_{\pi(i+k,j+l)}^YVW^{u_X}_{(i,j),(k,l)}(\eta\otimes \xi)$$
But, since $V$ intertwines, up to $\pi$, the projections on $\mathcal{F}_X$ and $\mathcal{F}_Y$, the above equation is equivalent to
$$U^Y_{\pi(i,j), \pi(k,l)}(V\eta\otimes V\xi) = VU^X_{(i,j),(k,l)}(\eta\otimes \xi)$$
Finally, notice that for all $e\in X(1,0)$, $f\in X(0,1)$ and $\eta\in X(i,j)$, 
$$VL^{(1,0)}_e\eta = VU^X_{(1,0),(i,j)}(e\otimes \eta) = U^Y_{\pi(1,0),\pi(i,j)}(Be\otimes V\eta) = L^{\pi(1,0)}_{Be}V\eta$$
$$VL^{(0,1)}_f\eta = VU^X_{(0,1),(i,j)}(f\otimes \eta) = U^Y_{\pi(0,1),\pi(i,j)}(Cf\otimes V\eta) = L^{\pi(0,1)}_{Cf}V\eta$$
Thus, $\phi(L_e^{(1,0)}) = L_{Be}^{\pi(1,0)} = VL^{(1,0)}_eV^\ast$ and $\phi(L_f^{(0,1)}) = L_{Cf}^{\pi(0,1)} = VL^{(1,0)}_fV^\ast$.
\end{proof}
\textit{Remark:} One may wonder at this point why have we inserted the `coordinate switch' monoid isomorphism into this discussion. Indeed, if we had not allowed a monoid isomorphism in both the definitions of subproduct system isomorphism, and of a $\mathbb{N}^2$-graded algebra isomorphism, the above proposition would still have been correct. The reason for this insertion will become clear only later. We will see that under certain conditions, a $\mathbb{N}$-graded isomorphism must be $\mathbb{N}^2$-graded, in the sense that may include the switch.\\

The next item on the agenda is to show that a $\mathbb{N}$-gradation of an isomorphism $\phi:\mathcal{A}_X\to\mathcal{A}_Y$ can be spotted by looking on its action on the character spaces. 

Recall, that given such $\phi$, we have a homeomorphism $\phi^\ast:\mathcal{M}(\mathcal{A}_Y)\to\mathcal{M}(\mathcal{A}_X)$ that is given by $\phi^\ast(\alpha):=\alpha\circ\phi$ on characters. By Theorem \ref{chars}, the character spaces of tensor algebras are homeomorphic to certain polyball varieties. Under this identification there is always one special character $\alpha_0\in\mathcal{M}(\mathcal{A}_X)$ which is identified with the zero vector in the Euclidean space. Recall again, that by Theorem \ref{chars}, a polyball variety is defined by a homogeneous ideal of polynomials. Hence, the zero vector is always present inside the variety, and obviously it remains fixed under different choices of bases for the subproduct system. So, the character $\alpha_0$ is well-defined, and we will call it the \textit{vacuum character}.
\begin{lem}\label{fourier-char}
If $X$ a subproduct system, and $\alpha_0\in\mathcal{M}(\mathcal{A}_X)$ is the vacuum character, then for all $T\in \mathcal{A}_X$ its $0$-th Fourier coefficient is given by $\widetilde{\Phi}_{0}(T) = \alpha_0(T)I$.
\end{lem}
\begin{proof}
Suppose $T\in \mathcal{A}_X$ is given. Proposition \ref{prod-char}(a) and the discussion throughout the previous section give us the information that, $\alpha_0\left(L^{(1,0)}_e\right)= \alpha_0\left(L^{(0,1)}_f\right)=0$ for all $e\in X(1,0)$ and $f\in X(0,1)$. Since $\alpha_0$ is a character, by Lemma \ref{finite-gen}, this means $\alpha_0\left(\widetilde{\Phi}_k(T)\right)=0$ for all $k\geq1$. But, since $\alpha_0$ is norm continuous, and $\sum_k \widetilde{\Phi}_k(T)$ Ces\`{a}ro-converges to $T$ in norm, we can conclude that $\alpha_0(T) = \alpha_0\left(\widetilde{\Phi}_0(T)\right)$. Recalling that the $0$-th Fourier coefficient must be a scalar multiple of $L^{(0,0)}_\Delta= I$, and that $\alpha_0\left(tI\right)=  t$ for all $t\in\mathbb{C}$, we obtain the statement of the lemma.
\end{proof}

The next lemma clarifies the connection between the vacuum character and the $\mathbb{N}$-gradation of $\mathcal{A}_X$.
\begin{lem}\label{long-ugly}
Let $X$ be a subproduct system, and $\alpha_0\in\mathcal{M}(\mathcal{A}_X)$ the vacuum character. Then, for all $T\in \mathcal{A}_X$ the following are equivalent:\\
\textbf{(a)} $\widetilde{\Phi}_0(T) = 0$ and $\widetilde{\Phi}_1(T) = 0$\\
\textbf{(b)} $T$ is the norm limit of sums of the form $\sum_{j=1}^{t_i} T_{ij}S_{ij}$ for some $\{T_{ij},S_{ij}\}\subset \mathcal{A}_X$ such that $\alpha_0(T_{ij})=\alpha_0(S_{ij})=0$.
\end{lem}
\begin{proof}
\textbf{(a)$\Rightarrow$(b)} \\
Suppose $\{e_i\}_{i=1}^m$ and $\{f_j\}_{j=1}^n$ are orthonormal bases for $X(1,0)$, $X(0,1)$, respectively. Then, for all $\eta\in X(k,l)$ with $k>0$, we can write $L^{(k,l)}_\eta = \sum_{i=1}^m L^{(1,0)}_{e_i}L^{(k-1,l)}_{\xi_i}$ for some $\{\xi_i\}$. Also, for $\eta\in X(0,l)$ with $l>0$, we can write $L^{(0,l)}_\eta = \sum_{j=1}^n L^{(0,1)}_{f_j}L^{(0,l-1)}_{\xi_j}$, for some $\{\xi_j\}$. But, we know that the series $\sum_{k=2}^\infty \widetilde{\Phi}_k(T)=\sum_{k=2}^\infty \sum_{i=0}^k L^{(i,k-i)}_{\eta_{i,k-i}}$ Ces\`{a}ro-converges to $T$ in norm.

Bearing in mind Lemma \ref{fourier-char} and the fact that $\widetilde{\Phi}_0\left(L^{(k,l)}_\eta\right)=0$ when $(k,l)\not=(0,0)$, we get the desired sequence from the Ces\`{a}ro sums. 

\textbf{(b)$\Rightarrow$(a)} \\
By Lemma \ref{fourier-char}, $\widetilde{\Phi}_0(T_{ij})= \widetilde{\Phi}_0(S_{ij}) =0$ for all $i,j$. That means for all $\eta\in X(k,l)$, for all $i,j$, we have 
$$T_{ij}\eta, S_{ij}\eta\in \bigoplus_{(0,0)\not=(s,t)\in\mathbb{N}^2} X(k+s,l+t)$$ 
So, 
$$T_{ij}S_{ij}\eta \in \bigoplus_{(s,t)\in\mathbb{N}^2\,s+t\geq2} X(k+s,l+t)$$
Thus, $p_{(k+1,l)}T_{ij}S_{ij} p_{(k,l)} = p_{(k,l+1)}T_{ij}S_{ij} p_{(k,l)} = 0$ where $p_{(s,t)}$ is the projection onto $X(s,t)$ in $B(\mathcal{F}_X)$, and that shows $\widetilde{\Phi}_1(T_{ij}S_{ij}) = 0$. Since $\widetilde{\Phi}_0, \widetilde{\Phi}_1$ are linear and norm-continuous, we can conclude the same for $T$.

\end{proof}

\begin{prop}\label{vacuum-to-n}
Let $X$, $Y$ be subproduct systems over $\mathbb{N}^2$, $\alpha_0^X, \alpha_0^Y$ the vacuum characters of the corresponding tensor algebras, and $\phi:\mathcal{A}_X\to\mathcal{A}_Y$ an isometric isomorphism. Then, $\phi^\ast(\alpha_0^Y)=\alpha_0^X$, if and only if, $\phi$ is $\mathbb{N}$-graded. 

In particular, when $\phi^\ast$ preserves the vacuum characters, we have the equality $\dim X(1,0)+\dim X(0,1) = \dim Y(1,0)+\dim Y(0,1)$.
\end{prop}
\begin{proof}
Suppose $\phi^\ast(\alpha_0^Y)=\alpha_0^X$. 
Fix some $\xi\in X(1,0)\oplus X(0,1)$ with $\|\xi\|=1$. By Lemma \ref{norm-eq}, $\|L^X_\xi\| = 1$. Since $\alpha_0(L^X_\xi)=0$, we know $\alpha_0(\phi(L^X_\xi))=0$. So, by Lemma \ref{fourier-char} we can write
$$\phi(L_{\xi}^X) = L^Y_{\eta} + T$$
for some $\eta\in Y(1,0)\oplus Y(0,1)$ and $T\in \mathcal{A}_Y$ such that $\widetilde{\Phi}_0(T) = \widetilde{\Phi}_1(T) = 0$. So, $\widetilde{\Phi}_1(L^Y_\eta+T) = L^Y_\eta$, and since $\widetilde{\Phi}_1$ is contractive, we have
$$\|L^Y_\eta\|\leq \|L^Y_\eta+T\| = \|\phi(L^X_\xi)\|=\|L^X_\xi\|=1$$
Now, by Lemma \ref{long-ugly}, $T$ is the limit of sums of the form $\sum_{j=1}^{t_i} T_{ij}S_{ij}$ in $\mathcal{A}_Y$, where $\alpha_0^Y(T_{ij})= \alpha_0^Y(S_{ij})=0$ for all $i,j$. So, since $\alpha_0^Y = (\phi^{-1})^\ast\left(\alpha_0^X\right)$ and $\phi^{-1}$ is a continuous isomorphism, we know that $\phi^{-1}(T)$ is a limit of similar sums in $\mathcal{A}_X$. Hence, by the second direction of Lemma \ref{long-ugly}, $\widetilde{\Phi}_0(\phi^{-1}(T)) = \widetilde{\Phi}_1(\phi^{-1}(T))=0$. So, in case $\phi^{-1}(T)\not=0$, we have the following contradiction.
$$1 = \|L^X_\xi\Delta_X\| < \|(L^X_\xi-\phi^{-1}(T))\Delta_X\|\leq \|L^X_\xi-\phi^{-1}(T)\| = \| L^Y_\eta\|\leq 1$$
Therefore, $T=0$, and $\phi(L_{\xi}^X) = L^Y_\eta$. 

The dimension equality is a direct consequence of the contents of the remark following Definition \ref{graded}.

Conversely, suppose $\phi$ is $\mathbb{N}$-graded. Then, for all $\xi\in X(1,0)\oplus X(0,1)$, $\alpha_0^Y(\phi(L^X_\xi)) = \alpha_0^Y(L^Y_{C\xi}) = 0$ ($C$ is some map into $Y(1,0)\oplus Y(0,1)$). Hence, $\alpha_0^Y\circ \phi = \alpha_0^X$.
\end{proof}

\subsection{Good subproduct systems}
Suppose $\phi$ is an isometric isomorphism between tensor algebras of subproduct systems. We have seen that the information that $\phi^\ast(\alpha_0^Y)=\alpha_0^X$ implies $\phi$ is $\mathbb{N}$-graded. On the other hand, it was demonstrated that if $\phi$ is $\mathbb{N}\times\mathbb{N}$-graded, then the underlying subproduct systems must be isomorphic. Suppose we could spot a subproduct system $X$ for which a $\mathbb{N}$-graded isomorphism of $\mathcal{A}_X$ must also be $\mathbb{N}\times\mathbb{N}$-graded. Then, when combining all propositions mentioned above, $X$ would have the property that if $\mathcal{A}_X$ and $\mathcal{A}_Y$ are isometrically isomorphic through a vacuum-character preserving isomorphism, then $X\cong Y$. In other words, for a class of such $X$'s the tensor algebra is close (up to the vacuum-character preservation property) to being a complete invariant. 

Indeed, in this subsection we identify one such class of subproduct systems.
\begin{defi}\label{good}
Suppose $X$ is a subproduct system, and $J$ is the $(m,n)$-homogeneous ideal of complex polynomials associated with $X$ as in Theorem \ref{chars}. We will say that $X$ is \textit{good} if the degrees $(p_m,p_n)$ of all polynomials $p\in J$ satisfy $p_m,p_n>0$.
\end{defi}
\textit{Remark:} Although the ideal $J$ in the above definition is constructed after choosing bases for $X(1,0)$ and $X(0,1)$ and is dependent on this choice, the degrees of its polynomials remain fixed under change of basis. One can see it when recalling the generators of $J$ as defined in Section 6.1. Under the notation of that section, the polynomials $p_{i,j}$ are always of degree $(1,1)$, while the degree of $q^x$ for some $x\in X(1,0)^{\otimes i}\otimes X(0,1)^{\otimes j}$ is $(i,j)$ regardless of choice of basis. \\ \\
For a subproduct system $X$ the property of being good can be equivalently put in terms of a geometrical attribute of its character space $\mathcal{M}(\mathcal{A}_X)$. Denote the following sets:
$$\mathcal{C}^{m,n}_1= \left\{(z,0)\in \mathbb{C}^m\times \mathbb{C}^n\,:\, \|z\|\leq 1\right\}\quad \mathcal{C}^{m,n}_2= \left\{(0,w)\in \mathbb{C}^m\times \mathbb{C}^n\,:\, \|w\|\leq 1\right\}$$
$$\mathcal{C}^{m,n} = \mathcal{C}^{m,n}_1\cup\mathcal{C}^{m,n}_2$$
A $(m,n)$-homogeneous complex polynomial of degree $(d_m,d_n)$ vanishes on $\mathcal{C}^{m,n}_1$ if and only if $d_n>0$, and vanishes on $\mathcal{C}^{m,n}_2$ if and only if $d_m>0$. Hence, for a $(m,n)$-homogeneous ideal $J$ of polynomials, $\mathcal{C}^{m,n}\subset \Omega^{m,n}(J)$ if and only if the degrees of all polynomials in $J$ have positive coordinates. Finally, this means $X$ is a good subproduct system if and only if $\mathcal{C}^{m,n}\subset \mathcal{M}(\mathcal{A}_X)$ (under the identification of the character space with a polyball variety, as given by Theorem \ref{chars}). 

We can equip the space $\mathbb{C}^m\times \mathbb{C}^n$ with the norm $\|(z,w)\|_{m,n}:= \max\{\|z\|,\|w\|\}$. Then, polyball varieties become intersections of algebraic varieties with the unit ball of this norm. The next lemma shows that the set $\mathcal{C}^{m,n}$ is invariant under linear isometries with respect to $\|\cdot\|_{m,n}$. This will be the main property that we will exploit in good subproduct systems.

\begin{lem}\label{baruch}
Suppose $u$ is a linear invertible transformation on $\mathbb{C}^n\times \mathbb{C}^n$ which maps a given homogeneous polyball variety $\Omega^{m,n}(I)$ onto a given homogeneous polyball variety $\Omega^{m,n}(J)$. Also, $u|_{\Omega^{m,n}(I)}$ is an isometry with respect to the norm $\|\cdot\|_{m,n}$. Then, 
$$u\left( \Omega^{m,n}(I)\cap \mathcal{C}^{m,n}\right)=\Omega^{m,n}(J)\cap \mathcal{C}^{m,n}$$
\end{lem}
\begin{proof}
Suppose $(x,v)\in \Omega^{m,n}(I)\setminus \mathcal{C}^{m,n}$. That is, $x,v\not=0$. To show that $\Omega^{m,n}(J)\cap \mathcal{C}^{m,n}\subset u\left( \Omega^{m,n}(I)\cap \mathcal{C}^{m,n}\right)$ it is enough to prove that $u(x,v)\not\in \mathcal{C}^{m,n}$. The other inclusion then follows when observing $u^{-1}$.

Note, that from homogeneity of $\Omega^{m,n}(I)$, we have $(\lambda_1 x,\lambda_2 v)\in\Omega^{m,n}(I)$ for all $\lambda_1,\lambda_2\in\mathbb{C}$ such that $\left\|(\lambda_1 x,\lambda_2 v)\right\|_{m,n}\leq1$. Write $u(x,0)=(y,z)$ and $u(0,v)=(w,t)$. Assume without loss of generality that $\|x\|\geq \|v\|$ and that $\|y\|\geq \|z\|$. Then, $\|y\|= \|(y,z)\|_{m,n} = \|(x,0)\|_{m,n} = \|x\|$. So, for all $\lambda\in \mathbb{C}$ with $|\lambda|=1$,
$$\|y\|^2 = \|u(x,\lambda v)\|_{m,n}^2 = \|(y+\lambda w, z+ \lambda t)\|_{m,n}^2\geq \|y+\lambda w\|^2 =$$ $$= \|y\|^2+2Re \lambda \langle w,y\rangle + \|w\|^2$$
Then, for a right choice of $\lambda$ we get that $\|y\|^2\geq \|y\|^2 + \|w\|^2$, hence, $w=0$, and $u(0,v) = (0,t)$. Of course, since $v\not=0$, we must have $t\not=0$.

Now, for an appropriate scalar $\mu$, we can have $(x,\mu v)\in \Omega^{m,n}(I)$ with $\|x\|\leq \|\mu v\|$, and $u(0,\mu v)= (0,\mu t)$. So, we can repeat the argument above, interchanging the coordinates, to conclude that $u(x,0) = (y,0)$ with $y\not=0$.

Thus, $u(x,v) = (y,t)\not\in \mathcal{C}^{m,n}$.
\end{proof}

\begin{prop}\label{n-to-n2}
Suppose $X$, $Y$ are subproduct systems over $\mathbb{N}^2$ with $\dim X(1,0)=m$, $\dim X(0,1)=n$, and $X$ is good. Suppose $\phi:\mathcal{A}_Y\to\mathcal{A}_X$ is an $\mathbb{N}$-graded isomorphism. Then, $\phi$ is $\mathbb{N}\times\mathbb{N}$-graded.
\end{prop}
\begin{proof}
We identify $\mathcal{M}(\mathcal{A}_X)$ with a homogeneous polyball variety relative to fixed bases of $X(1,0)$ and $X(0,1)$, and similarly for $\mathcal{M}(\mathcal{A}_Y)$ with bases for $Y(1,0)$ and $Y(0,1)$. It will also be convenient to consider vectors in $X(1,0)\oplus X(0,1)$ as vectors in $\mathbb{C}^{m}\times \mathbb{C}^{n}$ relative to the same basis (and the same for $Y$). Since $\phi$ is $\mathbb{N}$-graded, we have an invertible linear $A:Y(1,0)\oplus Y(0,1)\to  X(1,0)\oplus X(0,1)$ such that $\phi(L^Y_w) = L^X_{Aw}$. Denote $\alpha_z$ to be the character associated with $z\in \mathcal{M}(\mathcal{A}_X)$. Then, by the description of $\alpha_z$ (Theorems \ref{prod-char} and \ref{chars}), for all $z\in \mathcal{M}(\mathcal{A}_X)$ and $w\in \mathbb{C}^{m}\times\mathbb{C}^{n}$ we have
$$\langle w, \overline{\phi^\ast(z)}\rangle = \alpha_{\phi^\ast(z)}(L^Y_w) = \alpha_z(\phi(L^Y_w)) = \alpha_z(L^X_{Aw}) = \langle Aw, \overline{z}\rangle = \langle w, \overline{A^tz}\rangle$$
Thus, $\phi^\ast= A^t$ is an invertible linear transformation on $\mathbb{C}^{m}\times\mathbb{C}^n$.

Note, that because $\mathcal{M}(\mathcal{A}_X)$ is a homogeneous polyball variety, for all $z\in \mathcal{M}(\mathcal{A}_X)$ we have $\frac{z}{\|z\|_{m,n}}\in \mathcal{M}(\mathcal{A}_X)$. Hence,
$$\|\phi^\ast(z)\|_{m,n} = \|z\|_{m,n}\left\|\phi^\ast\left(\frac{z}{\|z\|_{m,n}}\right)\right\|_{m,n} \leq \|z\|_{m,n}$$
Since the same can be said about $(\phi^\ast)^{-1}$ we can conclude that $\phi^\ast$ acts as an isometry from $\mathcal{M}(\mathcal{A}_X)$ to $\mathcal{M}(\mathcal{A}_Y)$ with respect to the norm $\|\cdot\|_{m,n}$. 

Now, by Lemma \ref{baruch}, 
$$\phi^\ast\left(\mathcal{M}(\mathcal{A}_X)\cap \mathcal{C}^{m,n}\right) = \mathcal{M}(\mathcal{A}_Y)\cap \mathcal{C}^{m,n}$$
But, since $X$ is good, that means $\phi^\ast\left(\mathcal{C}^{m,n}\right)\subset \mathcal{C}^{m,n}$, and thus, $A^t\left(\mathcal{C}^{m,n}\right)\subset \mathcal{C}^{m,n}$. Hence, the invertible linear function $A^t$ must map the subspace $\mathbb{C}^{m}\times\{0\}$ into either $\mathbb{C}^{m}\times\{0\}$ or $\{0\}\times \mathbb{C}^{n}$, and the same for $\{0\}\times\mathbb{C}^{n}$. Therefore, $A$ can be described as
$$A = B\oplus C : X(1,0)\oplus X(0,1)\to Y(1,0)\oplus Y(0,1) \quad(\mbox{\small{or }} Y(0,1)\oplus Y(1,0))$$
\end{proof}

Finally, we have completed a chain of propositions which enables us to drop all assumptions about monoid gradations, and, instead, state the next result in a language related to the action on the character space.
\begin{thm}\label{complete-inv}
Suppose $X$, $Y$ are subproduct systems over $\mathbb{N}^2$ with $\dim X(1,0)=m$, $\dim X(0,1)=n$, and $X$ is good. Suppose there exists an isometric isomorphism $\phi:\mathcal{A}_Y\to \mathcal{A}_X$, and, furthermore, $\phi$ satisfies $\phi^\ast(\alpha_0^X)=\alpha_0^Y$, where $\alpha_0^X, \alpha_0^Y$ are the vacuum characters on the corresponding tensor algebras.

Then, $X$ is isomorphic to $Y$ as a subproduct system. Moreover, in this case $\phi$ is unitarily implemented, that is, there exists a unitary $V: \mathcal{F}_X\to\mathcal{F}_Y$ such that $\phi(T) = V^\ast TV$ for all $T\in\mathcal{A}_Y$. 
\end{thm}
\begin{proof}
This is just a consecutive application of Proposition \ref{vacuum-to-n}, Proposition \ref{n-to-n2} and Proposition \ref{n2-to-isom}. 
\end{proof}

\begin{examp}
Observe the class of good subproduct systems $X$ with $\dim X(1,0) = \dim X(0,1) =1 $. This class can be described in elementary terms. Since we can assume $X(i,j)\subset X(1,0)^{\otimes i}\otimes X(0,1)^{\otimes j}$ for all $i,j$, the dimension of $X(i,j)$ must be either $1$ or $0$, and the projection $p_{(i,j)}$ in the sense of Theorem \ref{standard-2} must be either the identity or zero. It can be easily verified that the only condition these projections need to satisfy in order for them to define a subproduct system is
$$p_{(i,j)}=0\quad\Rightarrow\quad  p_{(i+k,j)}=p_{(i,j+k)}=0\;\forall k\geq0$$ 
The fact that $X$ is good means only that $p_{(i,0)}=p_{(0,j)}=I$ for all $i,j$. Otherwise, we will have $z^i$ or $w^j$ as a polynomial in $J_X$ (the ideal from Theorem \ref{chars}), and that will contradict the good property. A generic example of such $X$ would look like
$$\begin{array}{c|ccccc}
& \vdots \\
& \mathbb{C} & \vdots &\vdots & \iddots \\
\vdots & \mathbb{C}& 0 & 0 &\cdots\\
2 & \mathbb{C} & \mathbb{C} & 0 & \cdots \\
1 & \mathbb{C} & \mathbb{C} & 0 & \cdots \\
0 & \mathbb{C} & \mathbb{C} & \mathbb{C} & \mathbb{C} & \cdots \\ \hline 
 & 0 & 1 & 2 & \cdots 

\end{array}$$
The intention of this table is that the entry on the $i$-th row and $j$-th column represents $X(i,j)$ as a vector space.\\ \\
We claim that every $X$ in this class uniquely defines its algebra.

Let us identify $\mathcal{M}(\mathcal{A}_X)$ with a polyball variety $\Omega^{1,1}(J_X)$. Suppose $J_X\not=\{0\}$. Then, there is a $(1,1)$-homogeneous $p\in J_X$. Since $X$ is good, it must be of the form $p(z,w) = az^kw^l$ with $k,l\geq1$ and $a\not=0$. Combining the facts that $X$ is good and that $p$ must vanish on $\Omega^{1,1}(J_X)$, we arrive at the equality 
$$\Omega^{1,1}(J_X) = \mathcal{C}^{1,1}\subset \mathbb{C}^2$$
So, if $\mathcal{A}_X\cong\mathcal{A}_Y$, then $\mathcal{M}(\mathcal{A}_Y)$ is homeomorphic to $\Omega^{1,1}(J_X)$ through the function $\phi^\ast$. But, the vacuum character which is identified with $0\in \Omega^{1,1}(J_X)$ is stable under homeomorphisms. Indeed, $0$ is the only point in $\mathcal{C}^{1,1}$ whose complement is not connected. Thus, $\phi^\ast(\alpha_0^Y)=\alpha_0^X$. So, by Theorem \ref{complete-inv}, $X\cong Y$. 

We finish the proof of the claim by showing there is only one subproduct, up to isomorphism, in this class, for which $J_X=\{0\}$. It is the product system with a trivial commutation relation. By that we mean $\dim X(i,j) = 1$ for all $i,j$, and the commutation relation $u:X(0,1)\otimes X(1,0)\to X(0,1)\otimes X(0,1)$ is given by $u(w\otimes z) = z\otimes w$. Indeed, $X(i,j)=\{0\}$ for some $i,j$, would have implied $z^iw^j\in J_X$. Also, any other commutation relation must have the form $u(w\otimes z)= \lambda z\otimes w$ for some $\lambda\not=1$. That would put $(\lambda-1)zw\in J_X$.

\end{examp}

\section{General isomorphisms and the orbit of the vacuum character}
The results of the previous section have shed some light on the isomorphism problem presented earlier. Yet, the discussion was heavily based on the assumption that $\phi^\ast(\alpha_0^Y)=\alpha_0^X$, where $\alpha_0^X, \alpha_0^Y$ are the vacuum characters on the corresponding subproduct systems. This assumption is quite restrictive. Already in the case of product systems, a large class of automorphisms of $\mathcal{A}_X$ which \textit{do not} fix the vacuum character was demonstrated in \cite{solel-unitary}.

Hence, we would like to drop that assumption now, and to inquire into what can be said about $\phi^\ast(\alpha_0^Y)$ in general. Perhaps more importantly, we are interested in what can be inferred about the relation of $X$ to $Y$ from the existence of a special character $\phi^\ast(\alpha_0^Y)$ in $\mathcal{M}(\mathcal{A}_X)$. As for the latter question, we arrive at two numerical values associated with a subproduct system, which are invariants of the isomorphism class of its tensor algebra. In other words, for $X$ and $Y$ as above, we will prove these two numerical values are equal.

For a subproduct system $X$, denote $m_X = \dim X(1,0)$ and $n_X = \dim X(0,1)$. The first invariant will be the dimension sum $m_X+n_X$. The second invariant, $k_X$, is a somewhat less approachable value. We will define it now, and aside from it being an invariant of the algebra isomorphism class, it will be useful to us throughout the forthcoming discussion. 

Recall that the character space of $\mathcal{A}_X$ can be identified with a polyball variety $\Omega^{m_X,n_X}(J_X)$, for a certain ideal of complex polynomials in $m_X+n_X$ variables. Define $k_X$ to be the minimal positive degree of polynomials in $J_X$. The composition of $J_X$ is determined in Theorem \ref{chars}, and although it depends on choice of bases, the degrees of polynomials in the ideal do not depend on that choice. (See the remark after the Definition \ref{good}.) Note, that from the above mentioned description of $J_X$, we know that $k_X\geq 2$ for all subproduct systems $X$.

Now, suppose again that $X$, $Y$ and $\phi$ are as above. Then, $\phi^\ast:\mathcal{M}(\mathcal{A}_Y)\to \mathcal{M}(\mathcal{A}_X)$ is a homeomorphism. Yet, we can view $\phi^\ast$ as a map from $\Omega^{m_Y,n_Y}(J_Y)$ to $\Omega^{m_X,n_X}(J_X)$, and now these spaces have a differential structure. Our methods will essentially try to show that $\phi^\ast$ preserves that differential structure as well. But, we avoid using deep concepts from complex algebraic (or analytic) geometry, and deal with the issue with elementary tools.

Denote $\mathbb{C}_n[t]$ to be the algebra of complex polynomials $\mathbb{C}[t]$ modulo the ideal generated by $t^n$. Note, that a complex polynomial $p$ equals zero when considered as an element in $\mathbb{C}_n[t]$ if and only if $p(0)=p'(0)=\ldots = p^{(n-1)}(0)=0$. The latter property will soon be of interest to us for certain polynomials. But, instead of looking directly on polynomial derivatives, we exploit the advantage of working with $\mathbb{C}_n[t]$ which has a notion of continuity for functions built into it. Indeed, after noting that this is a finite-dimensional algebra, we are free to put any norm on $\mathbb{C}_n[t]$ to end up with the same Euclidean topology. 

\begin{lem}\label{ck-homo}
Suppose $X$ is a subproduct system over $\mathbb{N}^2$, and $\{f_1,\ldots, f_{m_X}\}\subset X(1,0)$, $\{f_{m_X+1},\ldots, f_{m_X+n_X}\}\subset X(0,1)$ are orthonormal bases. Then,\\
\textbf{(a)} For all $\overline{\zeta} = (\zeta_1,\ldots, \zeta_{m_X+n_X})\in \mathbb{C}^{m_X+n_X}$ there exists a norm continuous unital homomorphism $\beta_{\overline{\zeta}}:\mathcal{A}_X\to \mathbb{C}_{k_X}[t]$ such that $\beta_{\overline{\zeta}}\left(L^X_{f_i}\right) = \zeta_it$ for all $1\leq i\leq m_X+n_X$.\\
\textbf{(b)} If $\lim_n \overline{\zeta}_n = \overline{\zeta}\in \mathbb{C}^{m_X+n_X}$, then for all $T\in \mathcal{A}_X$, the sequence $\beta_{\overline{\zeta}_n}(T)$ converges to $\beta_{\overline{\zeta}}(T)$. In other words, the mapping $\overline{\zeta}\mapsto \beta_{\overline{\zeta}}$ is continuous with respect to the strong operator topology on $B(\mathcal{A}_X, \mathbb{C}_{k_X}[t])$.
\end{lem}
\begin{proof}
\textbf{(a)} It was previously shown that the algebra $\widetilde{\mathcal{A}}_X$, which is generated (in the algebraic sense) by $\left\{I, L^X_{f_1},\ldots, L^X_{f_{m_X+n_X}}\right\}$ is norm dense in $\mathcal{A}_X$. Suppose $\overline{\zeta} = (\zeta_1,\ldots, \zeta_{m_X+n_X})\in \mathbb{C}^{m_X+n_X}$ is given. We will show that there is a unital homomorphism $\beta_{\overline{\zeta}}:\widetilde{\mathcal{A}}_X\to \mathbb{C}_{k_X}[t]$ such that 
$$\beta_{\overline{\zeta}}(L^X_{f_i}) = \zeta_it$$
Algebraically, $\widetilde{\mathcal{A}}_X$ is a quotient of the free algebra $\mathcal{P}_{m_X,n_X}\cong \mathbb{C}\langle a_1,\ldots, a_{m_X+n_X}\rangle$ by a non-commutative polynomial ideal $I_X$. Moreover, it is easy to see that $I_X$ is exactly the homogeneous ideal that corresponds to $X$ by Theorem \ref{non-comm-ideal}. Define a unital homomorphism $\beta_{\overline{\zeta}}:\mathcal{P}_{m_X,n_X}\to \mathbb{C}_{k_X}[t]$ by sending $a_i$ to $\zeta_it$. Then, for any $P\in I_X$, we have
$$\beta_{\overline{\zeta}}(P) = \epsilon(P)(\zeta_1t,\ldots,\zeta_{m_X+n_X}t)$$
where $\epsilon$ is the natural map from non-commutative polynomials into commutative ones (see Section 6.2). But, notice the simple fact that $J_X= \epsilon(I_X)$ (just observe the definitions of the two). So, either $\epsilon(P)$ is a homogeneous polynomial of degree greater-equal than $k_X$, and then $\epsilon(P)(\zeta_1t,\ldots,\zeta_{m_X+n_X}t)=0$ in $\mathbb{C}_{k_X}[t]$, or the degree of $\epsilon(P)\in J_X$ is less than $k_X$, and then by definition of $k_X$ we must have $\epsilon(P)=0$. Thus, we see that $\beta_{\overline{\zeta}}$ factors through the quotient by $I_X$, and we have the desired homomorphism from $\widetilde{\mathcal{A}}_X$.

To show that $\beta_{\overline{\zeta}}$ can be extended to $\mathcal{A}_X$ it is enough to assure this is a bounded homomorphism with respect to the norm on $\mathbb{C}_{k_X}[t]$. Indeed, it is clear that $\beta_{\overline{\zeta}}$ vanishes outside the finite dimensional summand \\$\bigoplus_{ i+j<k_X} \left\{L^{(i,j)}_x\::\:x\in X(i,j)\right\}$ of $\widetilde{\mathcal{A}}_X$, hence, bounded. 

\textbf{(b)} Suppose $\lim_n \overline{\zeta}_n= \overline{\zeta}\in \mathbb{C}^{m_X+n_X}$. For $T\in \widetilde{\mathcal{A}}_X$, $\beta_{\overline{\zeta}}(T)$ has coefficients which are just polynomial expressions in the coordinates of $\overline{\zeta}$. Thus,$\{\beta_{\overline{\zeta}_n}(T)\}$ clearly converges to $\beta_{\overline{\zeta}}(T)$, for all $T\in \widetilde{\mathcal{A}}_X$.

Now, we already saw that $\beta_{\overline{\zeta}_n}$ all vanish outside the same finite dimensional summand of $\widetilde{\mathcal{A}}_X$. This means $\|\beta_{\overline{\zeta}_n}\|$ can be bounded by the same polynomial expression in the coordinates of $\overline{\zeta}_n$, for all $n$. Since the sequence $\{\overline{\zeta}_n\}$ is bounded, we certainly can find a uniform bound such that $\|\beta_{\overline{\zeta}_n}\|\leq M$ for all $n$. 

Those two facts imply that $\{\beta_{\overline{\zeta}_n}(T)\}$ converges to $\beta_{\overline{\zeta}}(T)$ for all $T\in\mathcal{A}_X$.
\end{proof}

Before going through the main results of this section, we need another lemma which relates the structure of the algebra $\mathbb{C}_n[t]$ with differential properties of polynomials.

\begin{lem}\label{multi}
Let $p\in \mathbb{C}[z_1,\ldots z_n]$ be a given polynomial. Then the following are equivalent for all $\overline{z}\in\mathbb{C}^n$, and $k\in\mathbb{N}$:\\
\textbf{(a)} The vector $\overline{z}$ is a root of $p$ of multiplicity of at least $k$. \\
\textbf{(b)} In $\mathbb{C}_k[t]$, $p(\overline{z} + t\overline{\zeta} + t^2\overline{c_2}\ldots + t^{k-1}\overline{c_{k-1}})=0$ for all $\overline{\zeta},\overline{c_2},\ldots, \overline{c_{k-1}}\in\mathbb{C}^n$. \\
\textbf{(c)} In $\mathbb{C}_k[t]$, $p(\overline{z} + t\overline{\zeta} + t^2\overline{c_2}\ldots + t^{k-1}\overline{c_{k-1}})=0$ for all $\overline{\zeta}\in \mathbb{C}^n$ and some $\overline{c_2},\ldots, \overline{c_{k-1}}\in\mathbb{C}^n$ which may depend on $\overline{\zeta}$. \\
\textbf{(d)} There exists a non-empty open set $U\subset \mathbb{C}^n$ such that for all $\overline{\zeta}\in U$, there are $\overline{c_2},\ldots, \overline{c_{k-1}}\in\mathbb{C}^n$ s.t $p(\overline{z} + t\overline{\zeta} + t^2\overline{c_2}\ldots + t^{k-1}\overline{c_{k-1}})=0$ in $\mathbb{C}_k[t]$.
\end{lem}
\begin{proof}
For a fixed $\overline{\zeta} = (\zeta_1,\ldots, \zeta_n)\in\mathbb{C}^n$ and $\overline{c_2},\ldots \overline{c_{k-1}}\in \mathbb{C}^n$ define $\gamma(t) =\overline{z} + t\overline{\zeta} + t^2\overline{c_2}\ldots + t^{k-1}\overline{c_{k-1}}$ and the one-variable polynomial $f(t) = p(\gamma(t))$. Then, inductive reasoning shows that the $j$-th derivative of $f$ is given by
$$f^{(j)}(t) = \sum_{i_1,\ldots,i_j=1}^n \gamma'_{i_1}(t)\cdots \gamma'_{i_j}(t)\frac{\partial p}{\partial z_{i_1}\cdots \partial z_{i_j}}(\gamma(t)) + $$ $$+ \sum_{l=1}^{j-1}\sum_{i_1,\ldots,i_l=1}^n g^l_{i_1,\ldots,i_l}(t)\frac{\partial p}{\partial z_{i_1}\cdots \partial z_{i_l}}(\gamma(t))$$
where $\{g^l_{i_1,\ldots,i_t}\}$ are some polynomials.

\textbf{(a)$\;\Rightarrow\;$(b)}: Since $\overline{z}$ is root of multiplicity of at least $k$, we know that all partial derivatives of $p$ up to the $k-1$-th order vanish at $\overline{z}=\gamma(0)$. So, by the above formula we see that $f(0) = \ldots = f^{(k-1)}(0)=0$, regardless of the choices of $\overline{\zeta},\overline{c_2},\ldots,\overline{c_{k-1}}$. 

\textbf{(d)$\;\Rightarrow\;$(a)}: For each $\overline{\zeta}\in U$, fix $\gamma_\zeta(t) = \overline{z} + t\overline{\zeta} + t^2\overline{c_2}\ldots + t^{k-1}\overline{c_{k-1}}$ which satisfies the given condition. That is, $f_\zeta(t) = p(\gamma_\zeta(t))$ satisfies $f^{(j)}_\zeta(0) = 0$ for all $0\leq j\leq k-1$. That immediately gives $p(\overline{z}) = f_\zeta(0) = 0$.

We will prove that the partial derivatives of $p$ at $\overline{z}$ vanish up to the $k-1$-th order, by induction on the order of the derivative. For $j\leq k-1$, we assume all partial derivatives up to the $j-1$-th order vanish. Then, by the above formula we get for all $\overline{\zeta}= (\zeta_1,\ldots, \zeta_n)\in U$,
$$ 0 = f^{(j)}_\zeta(0) =  \sum_{i_1,\ldots,i_j=1}^n \gamma'_{i_1}(0)\cdots \gamma'_{i_j}(0)\frac{\partial p}{\partial z_{i_1}\cdots \partial z_{i_j}}(\gamma(0)) =$$ $$= \sum_{i_1,\ldots,i_j=1}^n \zeta_{i_1}\cdots \zeta_{i_k} \frac{\partial p}{\partial z_{i_1}\cdots \partial z_{i_j}}(\overline{z}) $$
We can view the right-hand side of the above equation as a complex polynomial in $\zeta_1,\ldots \zeta_n$. The equation says this polynomial vanishes on an open set, hence, its coefficients must all be zero. Since partial derivatives commute, those coefficients are integer multiples of the partial derivatives.

\end{proof}

\begin{thm}\label{atleast}
Suppose $X$ and $Y$ are subproduct systems over $\mathbb{N}^2$, and $\phi:\mathcal{A}_X\to\mathcal{A}_Y$ is a bounded isomorphism. Then, \\
\textbf{(a)} $m_X+n_X= m_Y+n_Y$\\
\textbf{(b)} The character $\phi^\ast(\alpha_0^Y)$, when seen as a vector in $\Omega^{m,n}(J_X)$, is a root of multiplicity of at least $k_Y$ of every polynomial in $J_X$.
\end{thm}
\begin{proof}
First, suppose that $m_X+n_X \leq m_Y+n_Y$. Choose orthonormal bases $\{f_1,\ldots, f_{m_Y}\}\subset Y(1,0)$ and $\{f_{m_Y+1},\ldots, f_{m_Y+n_Y}\}\subset Y(0,1)$, and in a similar manner choose bases $\{g_1,\ldots, g_{m_X+n_X}\}$ for $X$. By Lemma \ref{ck-homo}, for all $\overline{\zeta}\in\mathbb{C}^{m_Y+n_Y}$, there is a continuous unital homomorphism $\beta_{\overline{\zeta}}:\mathcal{A}_Y\to \mathbb{C}_{k_Y}[t]$ such that $\beta_{\overline{\zeta}}\left(L^Y_{f_i}\right) = t\zeta_i$ for all $i$. Define the following continuous homomorphisms:
$$h_1:\mathbb{C}_{k_Y}[t]\to \mathbb{C}\quad h_1(a+tb+ \ldots) = a$$
$$h_2:\mathbb{C}_{k_Y}[t]\to \mathbb{C}_2[t]\quad h_2(a+tb+\ldots)=a+tb$$
Then, $h_1\circ \beta_{\overline{\zeta}}$ equals the vacuum character for all $\overline{\zeta}\in \mathbb{C}^{m_Y+n_Y}$. So, $h_1\circ \beta_{\overline{\zeta}}\circ \phi$ is the character on $\mathcal{A}_X$ associated with $\phi^\ast(\alpha_0^Y)=(x_1,\ldots,x_{m_X+n_X})$. Hence, we can write for all $\overline{\zeta}\in \mathbb{C}^{m_Y+n_Y}$,
$$h_2\circ \beta_{\overline{\zeta}}\circ \phi (L^X_{g_i}) = x_i + ty_i$$
and define $\gamma(\overline{\zeta}):= (y_1,\ldots, y_{m_X+n_X})\in \mathbb{C}^{m_X+n_X}$. Note, that if $\gamma(\overline{\zeta}_1)=\gamma(\overline{\zeta}_2)$, then $h_2\circ \beta_{\overline{\zeta}_1}\circ \phi = h_2\circ \beta_{\overline{\zeta}_2}\circ \phi$ because they agree on the generators of $\mathcal{A}_X$. But, since $\phi$ is an isomorphism, that means $h_2\circ \beta_{\overline{\zeta}_1}=h_2\circ \beta_{\overline{\zeta}_2}$. That, in turn, clearly implies that $\overline{\zeta}_1=\overline{\zeta}_2$, and we can conclude that $\gamma$ is injective.

If $\overline{\zeta}_n\to\overline{\zeta}$, then when applying Lemma \ref{ck-homo}.(b) for $T = \phi(L^X_{g_i})$, we get that $\gamma(\overline{\zeta}_n)_i\to \gamma(\overline{\zeta})_i$. Thus, $\gamma$ is continuous.

It is a simple topological exercise (see, for example, \cite[Exercise 18.11]{topology} for guidance) to show that if we have an injective continuous function $\gamma:\mathbb{C}^{m_Y+n_Y} \to\mathbb{C}^{m_X+n_X}$, then we must have $m_Y+n_Y= m_X+n_X$ which proves \textbf{(a)}, and moreover, the image of $\gamma$ must be open. 

Note, that for all $\overline{\zeta}\in\mathbb{C}^{m_Y+n_Y}$,
$$\left(h_2\beta_{\overline{\zeta}}\phi(L^X_{g_1}),\ldots, h_2\beta_{\overline{\zeta}}\phi(L^X_{m_X+n_X})\right) = \phi^\ast(\alpha_0^Y)+t\gamma(\overline{\zeta}) \in \mathbb{C}_2[t]^{m_X+n_X}$$
which means there are $\overline{c}_2,\ldots, \overline{c}_{k_Y-1}\in \mathbb{C}^{m_X+n_X}$ such that
$$\left(\beta_{\overline{\zeta}}\phi(L^X_{g_1}),\ldots, \beta_{\overline{\zeta}}\phi(L^X_{g_{m_X+n_X}})\right) =$$ $$=
\phi^\ast(\alpha_0^Y)+t\gamma(\overline{\zeta}) + t^2\overline{c}_2 +\ldots + t^{k_Y-1}\overline{c}_{k_Y-1} \in \mathbb{C}_{k_Y}[t]^{m_X+n_X}$$
Thus, for every polynomial $p\in J_X$,
$$p\left(\phi^\ast(\alpha_0^Y)+t\gamma(\overline{\zeta}) + t^2\overline{c}_2 +\ldots + t^{k_Y-1}\overline{c}_{k_Y-1}\right)=$$ $$= \beta_{\overline{\zeta}}\circ\phi\left( p(L^X_{g_1},\ldots, L^X_{g_{m_X+n_X}})\right)= 0$$
Since the image of $\gamma$ has a non-empty interior, we can invoke Lemma \ref{multi} (d)$\;\Rightarrow\;$(a) to finish the proof.

As for the case with $m_X+n_X > m_Y+n_Y$, we apply \textbf{(a)} on the isomorphism $\phi^{-1}$ to get a contradiction.
\end{proof}

\begin{cor}\label{param}
If $\mathcal{A}_X$ is continuously isomorphic to $\mathcal{A}_Y$, then $k_X=k_Y$. 
\end{cor}
\begin{proof}
Suppose $k_X< k_Y$, and $\phi:\mathcal{A}_X\to\mathcal{A}_Y$ is a bounded isomorphism. Let $p\in I_X$ be a homogeneous polynomial of degree $k_X$. Then, by Theorem \ref{atleast} $\phi^\ast(0)$ must be a root of multiplicity no less than $k_Y$ of $p$. But, a polynomial has no roots of multiplicity bigger than its degree. Hence, a contradiction.

Otherwise, if $k_Y>k_X$, we apply the same argument on $\phi^{-1}$.
\end{proof}

In some cases, as in the next example, the last theorem can be conveniently applied to show that the only possible point in $\mathcal{M}(\mathcal{A}_X)$ to which $\alpha_0^Y$ can be sent by $\phi^\ast$ is in fact $\alpha_0^X$. Thus, if $X$ happens to be good, in combination with Theorem \ref{complete-inv} this will imply $X$ uniquely defines its algebra.

For example, suppose $X$ is a good subproduct system which satisfies the dimensions inequality: $\dim X(1,1)<\min\{m_X,n_X\}$. Such $X$'s surely exist: For one example, we can set $X(i,0)= X(1,0)^{\otimes i}$ and $X(0,i)= X(0,1)^{\otimes i}$ for all $i\geq1$. Then set $X(1,1)$ to be a small enough subspace of $X(1,0)\otimes X(0,1)$, and build the rest of the subproduct system using Proposition \ref{construct}.

\begin{cor}
A good subproduct system $X$ which satisfies $\dim X(1,1)<\min\{m_X,n_X\}$ uniquely defines its algebra.
\end{cor}
\begin{proof}
Suppose $\phi:\mathcal{A}_X\to\mathcal{A}_Y$ is an isometric isomorphism. Since $\dim X(1,1)< m_Xn_X$, we know $X(1,1)\not= X(1,0)\otimes X(0,1)$ and that means there are $(1,1)$-homogeneous polynomials in $J_X$ (the polynomial ideal associated with $X$). Hence, $k_X=2$, and by Corollary \ref{param} $k_Y=2$. So, by Theorem \ref{atleast}, $\phi^\ast(\alpha_0^Y)$, when considered as a vector in $\Omega^{m_X,n_X}(J_X)$, is a root of order at least 2 of all polynomials in $J_X$. Let $\{v_i\}_{i=1}^t$ be a basis for the vector space $\left(X(1,0)\otimes X(0,1)\right)\ominus X(1,1)$. From the bound on $\dim X(1,1)$, we know $t> m_Xn_X-n_X = (m_X-1)n_X$. We also can write
$$v_i = \sum_{j=1}^{n_X} h_{ij} \otimes f_j$$
where $\{f_j\}$ is a basis for $X(0,1)$ (relative to which we construct $J_X$), and $\{h_{ij}\}$ are vectors in $X(1,0)$. Define $V = span\{h_{ij}\}_{i,j=1}^{t, n_X}$, and notice that
$$n_X\cdot\dim V \geq \dim span\{v_i\}_{i=1}^t = t> (m_X-1)n_X$$
Hence, $\dim V= m_X$. Next, we recall from the construction of $J_X$ that each $v_i \perp X(1,1)$ defines a polynomial
$$q^{v_i}(\overline{z},\overline{w}) = \sum_{j=1}^{n_X} \langle \overline{z}, \widetilde{h_{ij}}\rangle w_j \in J_X$$ 
where $\widetilde{h_{ij}}$ refers to the vector in $\mathbb{C}^{m_X}$ which represents $h_{ij}$ under the chosen basis for $J_X$. Since $\phi^\ast(\alpha_0^Y)$ is a root of order at least 2 of all polynomials in $J_X$, we must have for all $i,j$,
$$0 = \frac{\partial q^{v_i}}{\partial w_j}\left(\phi^\ast(\alpha_0^Y)\right) = \left\langle \phi^\ast(\alpha_0^Y), \left(\widetilde{h_{ij}},0\right)\right\rangle \quad \Rightarrow \quad\phi^\ast(\alpha_0^Y) = (0,w) \in \Omega^{m_X,n_X}(J_X)$$
 The last implication above comes from the fact that $\{\widetilde{h_{ij}}\}$ span a $m_X$-dimensional space.
 
Now, since we also have $t> m_Xn_X-m_X$, we can run the symmetrical argument (switching $X(1,0)$ and $X(0,1)$) to obtain $\phi^\ast(\alpha_0^Y)= (z,0)\in \Omega^{m_X,n_X}(J_X)$. So, $\phi^\ast(\alpha_0^Y) =\alpha_0^X$, and by Theorem \ref{complete-inv} we have $X\cong Y$.
\end{proof}

\section{Acknowledgments}
I would like to thank my advisor, Baruch Solel, for suggesting this topic, providing guidance, and careful reading of many drafts.

\begin{thebibliography}{10}

\bibitem{popescu-distance}
A.~Arias and G.~Popescu.
\newblock Noncommutative interpolation and {P}oisson transforms.
\newblock {\em Israel J. Math.}, 115:205--234, 2000.

\bibitem{inclusion}
B.~V.~R. Bhat and M.~Mukherjee.
\newblock Inclusion systems and amalgamated products of product systems.
\newblock {\em Infin. Dimens. Anal. Quantum Probab. Relat. Top.}, 13(1):1--26,
  2010.

\bibitem{dav-distance}
K.~R. Davidson and D.~R. Pitts.
\newblock Nevanlinna-{P}ick interpolation for non-commutative analytic
  {T}oeplitz algebras.
\newblock {\em Integral Equations Operator Theory}, 31(3):321--337, 1998.

\bibitem{dav-rep}
K.~R. Davidson, S.~C. Power, and D.~Yang.
\newblock Dilation theory for rank 2 graph algebras.
\newblock {\em J. Operator Theory}, 63(2):245--270, 2010.

\bibitem{orr-subprod}
K.~R. Davidson, C.~Ramsey, and O.~M. Shalit.
\newblock The isomorphism problem for some universal operator algebras.
\newblock {\em Adv. Math.}, 228:167--218, 2011.

\bibitem{topology}
M.~Greenberg and J.~Harper.
\newblock {\em Algebraic topology: a first course}.
\newblock Mathematics lecture note series. Benjamin/Cummings Pub. Co., 1981.

\bibitem{gur-meyer}
M.~Gurevich and J.~Meyer.
\newblock Models for quotients of {H}ardy algebras.
\newblock {\em in preparation}.

\bibitem{highrank-one}
D.~W. Kribs and S.~C. Power.
\newblock The analytic algebras of higher rank graphs.
\newblock {\em Math. Proc. R. Ir. Acad.}, 106A(2):199--218 (electronic), 2006.

\bibitem{highrank}
A.~Kumjian and D.~Pask.
\newblock Higher rank graph {$C^\ast$}-algebras.
\newblock {\em New York J. Math.}, 6:1--20 (electronic), 2000.

\bibitem{solel-unitary}
S.~C. Power and B.~Solel.
\newblock Operator algebras associated with unitary commutation relations.
\newblock {\em J. Funct. Anal.}, 260(6):1583--1614, 2011.

\bibitem{solel-orr}
O.~M. Shalit and B.~Solel.
\newblock Subproduct systems.
\newblock {\em Doc. Math.}, 14:801--868, 2009.

\bibitem{ami}
A.~Viselter.
\newblock Covariant representations of subproduct systems.
\newblock {\em Proceedings of the London Mathematical Society}, 2010.

\end{thebibliography}

\end{document}